\documentclass[11pt]{amsart}

\usepackage{graphicx}
\usepackage{fullpage, amsmath, amsfonts, amssymb, amsthm}

\usepackage{enumitem}
\usepackage{xcolor}
\usepackage{caption}
\usepackage{tabularx}

\usepackage[hidelinks]{hyperref}
\usepackage{float}
\restylefloat{table}
\usepackage{mathtools}

\DeclarePairedDelimiter{\ceil}{\lceil}{\rceil}

\newcommand{\N}{\mathbb{N}}

\newcommand{\Z}{\mathbb{Z}}

\newcommand{\la}{\lambda}
\newcommand{\Mod}[1]{\ \left(\mathrm{mod}\ #1\right)}

\newtheorem{theorem}{Theorem}
\numberwithin{theorem}{section}
\numberwithin{equation}{section}
\newtheorem{corollary}[theorem]{Corollary}
\newtheorem{lemma}[theorem]{Lemma}

\newtheorem{proposition}[theorem]{Proposition}
\newtheorem{remark}[theorem]{Remark}
\newtheorem{conjecture}[theorem]{Conjecture}
\newtheorem{example}[theorem]{Example}

\newcommand{\qda}{q_d^{(a)}(n)}
\newcommand{\q}[3]{q_{#1}^{(#2)}\left(#3\right)}

\newcommand{\Qdb}{Q_d^{(b)}(n)}
\newcommand{\Qda}{Q_d^{(a)}(n)}
\newcommand{\Q}[3]{Q_{#1}^{(#2)}\left(#3\right)}
\newcommand{\Qdbdash}{Q_d^{(b,-)}(n)}
\newcommand{\Qdadash}{Q_d^{(a,-)}(n)}
\newcommand{\Qdash}[3]{Q_{#1}^{(#2,-)}\left(#3\right)}
\newcommand{\Qdbdashdash}{Q_d^{(b,-,-)}(n)}
\newcommand{\Qdadashdash}{Q_d^{(a,-,-)}(n)}
\newcommand{\Qdashdash}[3]{Q_{#1}^{(#2,-,-)}\left(#3\right)}

\newcommand{\Da}[3]{\Delta_{#1}^{(#2)}\left(#3\right)}
\newcommand{\Dab}{\Delta_d^{(a,b)}(n)}
\newcommand{\Daa}{\Delta_d^{(a)}(n)}
\newcommand{\Dsame}[1]{\Delta_d^{(#1,#1)}(n)}

\newcommand{\Ddasha}[3]{\Delta_{#1}^{(#2,-)}\left(#3\right)}
\newcommand{\Dabdash}{\Delta_d^{(a,b,-)}(n)}
\newcommand{\Daadash}{\Delta_d^{(a,-)}(n)}
\newcommand{\Ddash}[4]{\Delta_{#1}^{(#2,#3,-)}\left(#4\right)}
\newcommand{\Ddashdasha}[3]{\Delta_{#1}^{(#2,-,-)}\left(#3\right)}
\newcommand{\Dabdashdash}{\Delta_d^{(a,b,-,-)}(n)}

\newcommand{\Daadashdash}{\Delta_d^{(a,-,-)}(n)}
\newcommand{\Ddashdash}[4]{\Delta_{#1}^{(#2,#3,-,-)}\left(#4\right)}

\newcommand{\poch}[2]{(q^{#1};q^{#2})_{\infty}}
\newcommand{\pochneg}[2]{(-q^{#1};q^{#2})_{\infty}}

\newcommand{\fda}{f_d^{(a)}(q)}
\newcommand{\f}[1]{f_d^{(#1)}(q)}

\newcommand{\kda}{k_d^{(a)}(q)}
\renewcommand{\k}[1]{k_d^{(#1)}(q)}

\newcommand{\Kda}{\mathcal{K}_d^{(a)}(n)}
\newcommand{\K}[2]{\mathcal{K}_d^{(#1)}(#2)}

\newcommand{\Lda}{\mathcal{L}_d^{(a)}(n)}
\renewcommand{\L}[2]{\mathcal{L}_d^{(#1)}(#2)}

\newcommand{\Gda}{\mathcal{G}_d^{(a)}(n)}
\newcommand{\gda}{g_d^{(a)}(q)}

\begin{document}


\author{Adriana L. Duncan}
\address{Tulane University}
\email{aduncan3@tulane.edu}

\author{Simran Khunger}
\address{Carnegie Mellon University}
\email{skhunger@andrew.cmu.edu}

\author{Holly Swisher}
\address{Department of Mathematics, Kidder Hall 368, Oregon State University, Corvallis, OR 97331-4605}
\email{swisherh@math.oregonstate.edu}

\author{Ryan Tamura}
\address{University of California, Berkeley}
\email{rtamura1@berkeley.edu}

\title{Generalizations of Alder's Conjecture via a Conjecture of Kang and Park}

\subjclass[2010]{05A17, 11P81, 11P82, 11P84, 11F37}
\keywords{partitions, Rogers-Ramanujan identities, Alder's conjecture}

\thanks{This work was partially supported by the National Science Foundation REU Site Grant DMS-1757995, and Oregon State University.}

\begin{abstract}
Integer partitions have long been of interest to number theorists, perhaps most notably Ramanujan, and are related to many areas of mathematics including combinatorics, modular forms, representation theory, analysis, and mathematical physics. Here, we focus on partitions with gap conditions and partitions with parts coming from fixed residue classes.

Let $\Delta_d^{(a,b)}(n) = q_d^{(a)}(n) - Q_d^{(b)}(n)$ where $q_d^{(a)}(n)$ counts the number of partitions of $n$ into parts with difference at least $d$ and size at least $a$, and $Q_d^{(b)}(n)$ counts the number of partitions into parts $\equiv \pm b \Mod{d + 3}$. In 1956, Alder conjectured that $\Delta_d^{(1,1)}(n) \geq 0$ for all positive $n$ and $d$. This conjecture was proved partially by Andrews in 1971, by Yee in 2008, and was fully resolved by Alfes, Jameson and Lemke Oliver in 2011. Alder's conjecture generalizes several well-known partition identities, including Euler's theorem that the number of partitions of $n$ into odd parts equals the number of partitions of $n$ into distinct parts, as well as the first of the famous Rogers-Ramanujan identities. 

In 2020, Kang and Park constructed an extension of Alder's conjecture which relates to the second Rogers-Ramanujan identity by considering $\Delta_d^{(a,b,-)}(n) = q_d^{(a)}(n) - Q_d^{(b,-)}(n)$ where $Q_d^{(b,-)}(n)$ counts the number of partitions into parts $\equiv \pm b \Mod{d + 3}$ excluding the $d+3-b$ part. Kang and Park conjectured that $\Delta_d^{(2,2,-)}(n)\geq 0$ for all $d\geq 1$ and $n\geq 0$, and proved this for $d = 2^r - 2$ and $n$ even. 

We prove Kang and Park's conjecture for all but finitely many $d$. Toward proving the remaining cases, we adapt work of Alfes, Jameson and Lemke Oliver to generate asymptotics for the related functions. Finally, we present a more generalized conjecture for higher $a=b$ and prove it for infinite classes of $n$ and $d$.
\end{abstract}


\maketitle

\markboth{Duncan, Khunger, Tamura}{Generalizing Conjectures on Partitions with Congruence Relations and Difference Conditions}


\section{Introduction} \label{sec:introduction section}

A partition of a positive integer $n$ is a non-increasing sequence of positive integers, called parts, that sum to $n$.  Let $p(n \mid \text{condition})$ count the number of partitions of $n$ satisfying the specified condition. Euler famously proved that the number of partitions of a positive integer $n$ into odd parts equals the number of partitions of $n$ into distinct parts. Two other celebrated partition identities are those of Rogers and Ramanujan. The first Rogers-Ramanujan identity states that the number of partitions of $n$ into parts differing by $2$ is equal to the number of partitions of \(n\) into parts that are congruent to \(\pm 1 \Mod{5}\) and the second Rogers-Ramanujan identity states that the number of partitions of \(n\) into parts differing by \(2\) and with parts at least \(2\) is equal to the number of partitions of \(n\) into parts that are congruent to \(\pm 2 \Mod{5}\). These are encapsulated by
\begin{align*}
\sum_{n = 0}^{\infty} \frac{q^{n^2}}{(q;q)_n} &= \frac{1}{(q;q^5)_{\infty}(q^4;q^5)_{\infty}},\\
\sum_{n = 0}^{\infty} \frac{q^{n^2 + n}}{(q;q)_n} &= \frac{1}{(q^2;q^5)_{\infty}(q^3;q^5)_{\infty}},
\end{align*}
where $(a;q)_0=1$, and $(a;q)_n := \prod_{k=0}^{n-1} (1-aq^k)$ for $n\geq 1$, where $n=\infty$ is also allowed.

Motivated by these identities, Schur found that the number of partitions of $n$ into parts differing by $3$ or more among which no two consecutive multiples of 3 appear is equal to the number of partitions of $n$ into parts congruent to $\pm 1 \pmod{6}$. 

After showing that no other such partition identities can exist, in 1956 Alder \cite{alder_nonexistence_1948, alder_research_1956} made a claim about a generalization of similar partition inequalities. Alder's conjecture states that the number of partitions of \(n\) into parts with difference of at least \(d\) is greater than or equal to than the number of partitions of \(n\) into parts congruent to \(\pm 1 \Mod{d + 3}\). Notice that the Euler, first Rogers-Ramanujan, and Schur identities are special cases of Alder's Conjecture. In 1971, Andrews \cite{andrews_partition_1971} proved Alder's conjecture for $n \geq 1$ and \(d = 2^r - 1\), \(r \geq 4\). In 2004 and 2008, Yee \cite{yee_alders_2008, yee_partitions_2004} proved the conjecture for $n \geq 1$, \(d \geq 32\), and \(d = 7\). In 2011, Alfes, Jameson, and Lemke Oliver \cite{alfes_proof_2011} proved Alder's Conjecture for $n \geq 1$ and \(4 \leq d \leq 30, d \neq 7,15\), thus completely resolving the conjecture.

In 2020, Kang and Park \cite{kang_analogue_2020} investigated how to generalize Alder's conjecture further by incorporating the second Rogers-Ramanujan identity. Kang and Park compared the partition functions
\begin{align*}
    &\qda := p(n | \text{ parts} \geq a \text{ and parts differ by at least } d), \\
    &\Qdb := p(n | \text{ parts} \equiv \pm b \Mod{d+3}),
\end{align*}
by defining the difference function
\begin{align*}
    \Dab &:= \qda - \Qdb \nonumber, \\
    \Delta_d^{(a)}(n) &:= \Dsame{a}.
\end{align*}
\begin{remark}
    Notice that
    \begin{align*}
        \text{Euler's identity } \iff \Da{1}{1}{n} &= 0 \text{ for all } n \geq 1
        \\
        \text{Rogers-Ramanujan (\(1^{st}\) identity)} \iff \Da{2}{1}{n} &= 0 \text{ for all } n \geq 1
        \\
        \text{Rogers-Ramanujan (\(2^{nd}\) identity)} \iff \Da{2}{2}{n} &= 0 \text{ for all } n \geq 1
        \\
        \text{Schur's identity } \implies \Da{3}{1}{n} &\geq 0 \text{ for all } n \geq 1.
    \end{align*}
\end{remark}

Using Kang and Park's notation, Alder's conjecture (now theorem) can be stated as  
\begin{equation}\label{eq:alderthm}
\Da{d}{1}{n} = \q{d}{1}{n} - \Q{d}{1}{n} \geq 0 
\end{equation}
for $d,n\geq 1$.

Kang and Park were interested in finding an analog of Alder's conjecture for the second Rogers-Ramanujan identity. However, by observing the data, they found that 
\[
    \Da{d}{2}{n} < 0 \text{ for some choices of } d,n\geq 1.
\]
In order to find a suitable analog, Kang and Park modified \(\Q{d}{2}{n}\) by defining for $d,n\geq 1$,
\begin{align*}
    \Qdash{d}{2}{n} &:= p(n\, | \text{ parts} \equiv \pm 2 \Mod{d+3}, \text{ excluding the part } d+1), \\
    \Ddasha{d}{2}{n} & := \Ddash{d}{2}{2}{n} := \q{d}{2}{n} - \Qdash{d}{2}{n},
\end{align*}
and presented the following conjecture.

\begin{conjecture}[Kang, Park \cite{kang_analogue_2020}, 2020]\label{conj:KPconj}
    For all $d,n\geq 1$,
    \[
        \Ddasha{d}{2}{n} \geq 0.
    \]
\end{conjecture}
Kang and Park \cite{kang_analogue_2020} proved Conjecture \ref{conj:KPconj} when $d=2$ or $d=2^s-2$ for any positive integer $s\geq 5$, and positive even integer $n$.  By developing a new way of comparing these partition functions, we prove Kang and Park's conjecture except for the cases $d = 1$ and $3 \leq d \leq 61$\footnote{The \(d = 2\) case is simply the second Rogers-Ramanujan identity.}. 

\begin{theorem}\label{thm:final kp thm}
    For \(d\geq 62\) and $n\geq 1$, 
    \[
        \Ddasha{d}{2}{n} \geq 0.\\
    \]
\end{theorem}

It is natural to ask whether Conjecture \ref{conj:KPconj} can be generalized to consider higher \(a = b\). For $1 \leq b \leq d + 2$, define
\begin{align*}
    \Qdash{d}{b}{n} &:= p(n\, | \text{ parts} \equiv \pm b \Mod{d+3}, \text{ excluding the part } d+3 - b), \\
    \Dabdash &:= \qda - \Qdbdash \text{ and } \Ddasha{d}{a}{n} := \Ddash{d}{a}{a}{n}.
\end{align*}
We conjecture the following.
\begin{conjecture}\label{conj:schurconj}
For all $d,n\geq 1$,
\[
    \Ddasha{d}{3}{n} \geq 0.
\]
\end{conjecture}

When $a\geq 4$, $\Daadash$ is not always nonnegative. Surprisingly, excluding just one additional part in the definition of $\Qdadash$ allows us to generalize Conjecture \ref{conj:KPconj} to arbitrary $a$.  For $1 \leq b \leq d + 2$, we define\footnote{For a discussion on why the exclusion of the \(b\) and \(d + 3 - b\) parts are necessary, see Section \ref{sec:data and dash discussion section}.}
\begin{align*}
\Qdbdashdash &:= p(n \,| \text{ parts} \equiv \pm b \Mod{d + 3}, \text{excluding the parts } b \text{ and } d+3-b), \\
\Dabdashdash &:= \qda - \Qdbdashdash \text{ and } \Ddashdasha{d}{a}{n} := \Ddashdash{d}{a}{a}{n},
\end{align*}
and posit the following conjecture.

\begin{conjecture}\label{conj:Generalized KP Conjecture}
Let $a,d$ be positive integers with $1 \leq a \leq d+2$.  Then for all $n \geq 1$,
\[
\Ddashdasha{d}{a}{n} \geq 0. 
\]
\end{conjecture}

\noindent We observe that Conjecture \ref{conj:Generalized KP Conjecture} is true when $a = 1$ and $d,n \geq 1$ by \eqref{eq:alderthm}.  Moreover, when $a = 2$, $d \geq 62$, and $n \geq 1$, Conjecture \ref{conj:Generalized KP Conjecture} follows from Theorem \ref{thm:final kp thm}.  

We now state three partial results toward a proof of Conjecture \ref{conj:Generalized KP Conjecture} for $a \geq 3$.  The first uses methods from the proof of Theorem \ref{thm:final kp thm}.

\begin{theorem} \label{thm:stringarbitrary}
Let $a \geq 3$, and $d\geq 31a - 3$, such that $a$ divides $d+3$. Then for all $n \geq 1$, 
\[ 
\Ddashdasha{d}{a}{n} \geq 0. 
\]
\end{theorem}

\noindent The next two results require the following definition.  Given fixed $a,d\geq 1$, define $r$ (dependent on $d$ and $a$) to be the largest integer such that 
\begin{equation}\label{eq:r def}
2^r - 2^{a-1} \leq d.
\end{equation}
By generalizing methods of Andrews \cite{andrews_partition_1971}, we prove the following, which generalizes results of Andrews \cite[Theorem 4]{andrews_partition_1971} as well as Kang and Park \cite[Theorem 1.3]{kang_analogue_2020}.  

\begin{theorem} \label{thm:andrewsarbitrary}
Suppose $1 \leq a \leq d+2$, where $d = 2^{s}-2^{t}$, with $t \geq \log_2(a)$ and $s \geq t + 4$.  Then for all $n\geq 1$, 
\[
\Ddashdasha{d}{a}{2^t n} \geq 0.
\]
\end{theorem}

\noindent By generalizing methods of Yee \cite{yee_alders_2008},  We futher prove the following.

\begin{theorem} \label{thm:yeearbitrary}
Let $a\geq 3$, and $d=2^{a-1}m$ where $m\geq 31$ and $m \not= 2^{r-a+1} - 1$. Then for all $n \geq 2m + 2^{r-a+1}+1$, 
\[
\Ddashdasha{d}{a}{2^{a-1}n} \geq 0.
\]   
\end{theorem}

Finally, we investigate and generalize asymptotic results of Andrews \cite{andrews_partition_1971}, as well as Alfes, Jameson, and Lemke Oliver \cite{alfes_proof_2011}. 

\begin{theorem}\label{thm:asymptotic result}
Let $a\geq 1$ and $d \geq 4$ such that $a < \frac{d + 3}{2}$ and $\gcd(a,d+3)=1$. Then for all $n\geq 1$,
\[
\lim_{n \to \infty} \Daadashdash = + \infty.
\]
\end{theorem}

We note that it is useful to keep in mind the following observation.  For $1 \leq a \leq d+2$ and $n\geq 1$,
\begin{equation} \label{rem:thing}
\Ddashdasha{d}{a}{n} \geq \Ddasha{d}{a}{n} \geq \Da{d}{a}{n}.
\end{equation}

We now outline the rest of the paper.  In Section \ref{sec:preliminaries section}, we generalize a theorem of Andrews \cite{andrews_partition_1971} and establish some other useful lemmas. In Section \ref{sec:modified alders section}, we prove two modified versions of Alder's conjecture that will be used in the proofs of Theorem \ref{thm:final kp thm} and Theorem \ref{thm:stringarbitrary}. In Section \ref{sec:kang and park conjecture section}, we prove Theorem \ref{thm:final kp thm} by considering four cases based on the parities of $n$ and $d$, and present a partial proof of Conjecture \ref{conj:schurconj} in Proposition \ref{prop:schurprop}. In Section \ref{sec:generalized kang and park conjecture section}, we modify methods from the proof Theorem \ref{thm:final kp thm} to prove Theorem \ref{thm:stringarbitrary}, adapt methods of Andrews \cite{andrews_partition_1971} to prove Theorem \ref{thm:andrewsarbitrary}, and generalize methods of Yee \cite{yee_alders_2008} to prove Theorem \ref{thm:yeearbitrary}. In Section \ref{sec:asymptotics section}, we prove Theorem \ref{thm:asymptotic result} as well as analog a method of Alfes, Jameson, and Lemke Oliver \cite{alfes_proof_2011} to obtain an asymptotic expression for $\Daa$ with explicit error term.  We then discuss how asymptotics could be useful to prove remaining cases of Conjecture \ref{conj:KPconj}. Finally, in Section \ref{sec:data and dash discussion section}, we conclude with some remarks about our choice of definitions.

\section{Preliminaries}\label{sec:preliminaries section}

In this section, we first develop generating functions for our partition functions, and then establish some important lemmas that we will employ in later sections.   As in Yee \cite{yee_alders_2008}, we occasionally denote the coefficient of $q^n$ in an infinite series $s(q)$ as $[q^n](s(q))$.

For $a,d\geq 1$, the generating function for $\qda$ is
\begin{equation}
\sum_{n = 0}^{\infty} \qda q^n = \sum_{k = 0}^{\infty} \frac{q^{d{k \choose 2} + ka}}{(q;q)_{k}}. \label{eq:qdagenfunc}
\end{equation}
Equation \eqref{eq:qdagenfunc} appears in work of Alder \cite{alder_nonexistence_1948}, and the case where \(a = 1\) is proved combinatorially in \cite{lehmer_two_1946}, which can be easily generalized to $a\geq 1$.

As stated in Alfes, Jameson, and Lemke Oliver \cite{alfes_proof_2011}, the generating function for $\Qdb$ is
\begin{equation}\label{eq:Qdbgenfunc}
\sum_{n = 0}^{\infty} \Qdb q^n = \frac{1}{(q^{d + 3 - b};q^{d + 3})_{\infty}(q^{b};q^{d + 3})_{\infty}}. 
\end{equation}
From \eqref{eq:Qdbgenfunc} we easily obtain the generating function for $\Qdbdash$ by removing the term that generates the part $d + 3 - b$.  For $1\leq b \leq d+2$, the generating function for $\Qdbdash$ is thus given by
\begin{equation}
\sum_{n = 0}^{\infty} \Qdbdash q^n =
\begin{cases}
\frac{1}{(q^{2d + 6 - b};q^{d + 3})_{\infty}} &\text{ for } b = \frac{d + 3}{2},\\
\frac{1}{(q^{2d + 6 - b};q^{d + 3})_{\infty}(q^{b};q^{d + 3})_{\infty}} &\text{ otherwise.} \label{eq:Qdbdashgenfunc}
\end{cases}
\end{equation}
Similarly we can construct the generating function for $\Qdbdashdash$, which for $1\leq b \leq d+2$, is given by
\begin{equation}\label{eq:Qdbdashdashgenfunc}
\sum_{n = 0}^{\infty} \Qdbdashdash q^n =
\begin{cases}
\frac{1}{(q^{2d + 6 - b};q^{d + 3})_{\infty}} &\text{ for } b = \frac{d + 3}{2}, \\
\frac{1}{(q^{2d + 6 - b};q^{d + 3})_{\infty}(q^{d + 3 + b};q^{d + 3})_{\infty}} &\text{ otherwise.}\\ 
\end{cases}
\end{equation}

Andrews \cite{andrews_partition_1971} uses the following theorem in his proof of certain cases of Alder's conjecture.  Write
\[
\rho(R;n):= p(n \mid \text{parts are from the set } R).
\]

\begin{theorem}[Andrews \cite{andrews_partition_1971}, 1971 ]\label{thm:andrews ST}
Let $S = \{x_i\}_{i = 1}^{\infty}$ and $T = \{y_i\}_{i = 1}^{\infty}$ be two strictly increasing sequences of positive integers such that $y_1 = 1$ and $x_i \geq y_i$ for all $i$. Then for all $n >0$,
\[
\rho(T;n) \geq \rho(S;n).
\]
\end{theorem}

\noindent Yee \cite{yee_alders_2008} employs this theorem in her work on Alder's conjecture. Kang and Park \cite{kang_analogue_2020} present a generalization of Theorem \ref{thm:andrews ST} which we generalize further below.  

\begin{lemma}\label{lem:STlemma}
Let $S = \{x_{i}\}_{i =1}^{\infty}$ and $T = \{y_{i}\}_{i=1}^{\infty}$ be strictly increasing sequences of positive integers such that $y_{1} = m$, $m$ divides each $y_{i}$, and $x_{i} \geq y_{i}$ for all $i$. Then for all $n\geq 1$,
\[
\rho(T;mn) \geq \rho(S;mn). 
\]
\end{lemma}

\begin{proof}
Let $\widetilde{S}_{mn}$, $\widetilde{T}_{mn}$ be the sets of partitions of $mn$ with parts coming from $S$, $T$, respectively.  We construct an injection $\varphi: \widetilde{S}_{mn} \to \widetilde{T}_{mn}$, similar to Yee's \cite{yee_alders_2008} proof of Theorem \ref{thm:andrews ST}.

Let $\lambda$ be a partition in $\widetilde{S}_{mn}$. Define $e_i$ to be the number of times $x_i$ occurs in $\lambda$, so that $\lambda$ can be described by $mn = \sum_{i=1}^{\infty} e_ix_i$, which is clearly divisible by $m$. Since each $y_{i}$ is divisible by $m$, we can conclude that 
\[
\sum_{i=1}^{\infty}(x_i - y_i)e_i
\] 
is divisible by $m$. 

Define the sequence $\{f_{i}\}_{i=1}^{\infty}$ by
\[  
f_{i} = \begin{cases}
e_1 + \frac{1}{m} \sum_{j=1}^{\infty}(x_j - y_j)e_j, & i = 1\\
e_i, & i > 1.
\end{cases}
\]
From this we can define $\varphi(\lambda)$ to be the partition consisting of parts $y_i$ each occurring $f_i$ times.  Since $y_1=m$, 
\[
\sum_{i=1}^{\infty}f_iy_i = e_{1}y_{1} + \sum_{i=1}^{\infty}(x_i -y_i)e_i + \sum_{i=2}^{\infty}e_iy_i = \sum_{i=1}^{\infty}e_ix_i = mn,
\] 
so $\varphi(\lambda)$ is indeed a partition of $mn$ and thus $\varphi(\lambda) \in \widetilde{T}_{mn}$ as desired.   

It remains to show that $\varphi$ is injective. Let $\lambda, \lambda' \in \widetilde{S}_{mn}$ such that $\lambda$ is described by $mn = \sum_{i=1}^{\infty} e_ix_i$, and $\lambda'$ is described by $mn = \sum_{i=1}^{\infty} g_ix_i$.  If $\varphi(\lambda) =\varphi(\lambda')$, then we can immediately conclude $e_i=g_i$ for all $i \geq 2$ by the definition of $\varphi$.  Moreover, we must have
\[
e_1 + \frac{1}{m} \sum_{i=1}^{\infty}(x_i - y_i)e_i = g_1 + \frac{1}{m} \sum_{i=1}^{\infty}(x_i - y_i)g_i,
\]
and thus $e_1=g_1$ as well. So $\lambda = \lambda'$ and we conclude that $\varphi$ is injective. 
\end{proof} 

The following application of Lemma \ref{lem:STlemma} will be used in the proofs of Theorems \ref{thm:andrewsarbitrary} and \ref{thm:yeearbitrary}. 

\begin{corollary}\label{cor:STgeneral}
Let $a\geq 1$, $t\geq \log_2(a)$, $s\geq t+4$, $d\geq 2^s-2^t$ such that $2^t$ divides $d$, and $d-3 \leq b \leq d$.  Then
\[
\rho(T_{t,s,d}; 2^t n) \geq \rho(S_{a,b}; 2^t n),
\]
where
\begin{align*}
T_{t,s,d} & := \{y \in \N \mid y \equiv 2^t, d+2^{t+1}, \ldots, d+ 2^{s-1} \!\!\! \pmod{2d} \}, \\
S_{a,b} & := \{x\in \N \mid x \equiv \pm a \!\!\! \pmod{b+3} \} \backslash \{ a, b+3-a\}.
\end{align*}
\end{corollary}

\begin{proof}
Since $d$ is divisible by $2^t$, so is each element of $T_{t,s,d}$.  Moreover, the smallest element of $T_{t,s,d}$ is $2^t$, so if we can show that the $i$th term of $T_{t,s,d}$ is at most the size of the $i$th term of $S_{a,b}$ (ordering by size), then $\rho(T_{t,s,d}; 2^t n) \geq \rho(S_{a,b}; 2^t n)$ will follow directly from Lemma \ref{lem:STlemma}.

By hypothesis, $2^t\leq 2^s - 2^t \leq d$.  It follows that the elements of $T_{t,s,d}$ ordered by size across rows are as shown in Figure \ref{GeneralT}.

\begin{figure}[h!]\caption{Elements of $T_{t,s,d}$ ordered by size across rows}\label{GeneralT} 
\[
\begin{array}{|c|c|c|c|}
\hline
2^t & d + 2^{t+1} & \cdots & d + 2^{s-1} \\
\hline
2d + 2^t & 3d + 2^{t+1} & \cdots & 3d + 2^{s-1} \\ 
\hline
\vdots & \vdots & \vdots & \vdots \\
\hline
(2j-2)d + 2^t & (2j-1)d + 2^{t+1} & \cdots & (2j-1)d + 2^{s-1} \\ 
\hline
\vdots & \vdots & \vdots & \vdots \\
\hline
\end{array}
\]
\end{figure}

Notice that the first column is independent of $s$, and there are $s-t$ elements in each row.  Decreasing $s$ will yield fewer columns, so we can conclude that when ordered by size, the $i$th element of $T_{t,s+1,d}$ is less than the $i$th element of $T_{t,s,d}$ for any $s \geq t+4$.  Thus to finish the proof it suffices to show that $i$th term of $T_{t,t+4,d}$ is at most the size of the $i$th term of $S_{a,b}$.  Here, $T_{t,t+4,d}$ has elements ordered by size across rows as given in Figure \ref{SpecificT}.

\begin{figure}[h!]\caption{Elements of $T_{t,t+4,d}$ ordered by size across rows}\label{SpecificT} 
\[
\begin{array}{|c|c|c|c|}
\hline
2^t & d + 2^{t+1} & d + 2^{t+2}  & d + 2^{t+3} \\
\hline
2d + 2^t & 3d + 2^{t+1}  & 3d+ 2^{t+2} & 3d + 2^{t+3}\\ 
\hline
\vdots & \vdots & \vdots & \vdots \\
\hline
(2j-2)d + 2^t & (2j-1)d + 2^{t+1}  & (2j-1)d + 2^{t+2} & (2j-1)d + 2^{t+3} \\ 
\hline
\vdots & \vdots & \vdots & \vdots \\
\hline
\end{array}
\]
\end{figure}

By hypothesis, $a\leq 2^t$ and $d \geq 15\cdot 2^t$.  Thus for $k$ a positive integer, $k(d+3) + a < (k+1)(d+3) -a$.  It follows that the elements of $S_{a,b}$ are given in order of size in Figure \ref{Selements}.

\begin{figure}[h!]\caption{Elements of $S_{a,b}$ ordered by size across rows}\label{Selements} 
\[
\begin{array}{|c|c|c|c|}
\hline
(b+3) + a  & 2(b+3) - a & 2(b+3) + a & 3(b+3) - a \\
\hline
3(b+3) + a & 4(b+3) - a & 4(b+3) + a & 5(b+3) - a \\ 
\hline
\vdots & \vdots & \vdots & \vdots \\
\hline
(2j-1)(b+3) + a & 2j(b+3) - a & 2j(b+3) + a & (2j+1)(b+3) - a \\ 
\hline
\vdots & \vdots & \vdots & \vdots \\
\hline
\end{array}
\]
\end{figure}

Let $x_i$ and $y_i$ be the $i$th smallest elements of $S_{a,b}$ and $T_{t,t+4,d}$, respectively.  We now prove that $x_i \geq y_i$ by considering the corresponding elements in the $j$th rows of the tables in Figures \ref{Selements} and \ref{SpecificT}.  Each of the four cases follows in a straightforward manner from the facts that $d \geq 15\cdot 2^t$, $d-3\leq b \leq d$, and $a\leq 2^t$, which concludes the proof.  
\end{proof}


We now present two lemmas that we use in the proofs of Theorem \ref{thm:final kp thm}, Theorem \ref{thm:stringarbitrary}, and later Proposition \ref{prop:schurprop}.

\begin{lemma}\label{lem:qstarlemma}
Let $a, d \geq 1$, and $n\geq d+2a$.  Then,
\[ 
\qda \geq \q{\ceil*{\frac{d}{a}}}{1}{\ceil*{\frac{n}{a}}}.
\]
\end{lemma}

\begin{proof}
For a positive integer $x$, write $\hat{x}$ to denote the least residue of $-x \pmod{a}$ so that $\ceil*{\frac{x}{a}} = \frac{x + \hat{x}}{a}$.  Fix $n\geq d+2a$, and define $\q{d+\hat{d}}{a}{n+\hat{n}}^*$ to count the number of partitions of $n+\hat{n}$ into parts $\geq a$, each a multiple of $a$, with difference between parts $\geq d+\hat{d}$.  Then it is straightforward to see that 
\[
\q{d+\hat{d}}{a}{n+\hat{n}}^* =  \q{\frac{d + \hat{d}}{a}}{1}{\frac{n + \hat{n}}{a}} = \q{\ceil*{\frac{d}{a}}}{1}{\ceil*{\frac{n}{a}}},
\]
by the natural bijection between the relevant sets of partitions obtained by dividing (or multiplying) the parts by $a$.  Thus it suffices to show that $\qda \geq \q{d+\hat{d}}{a}{n+\hat{n}}^*$. 

Let $X(n)$ be the set of partitions of $n$ counted by $\qda$ and $Y(n+\hat{n})$ the set of partitions of $n+\hat{n}$ counted by $\q{d+\hat{d}}{a}{n+\hat{n}}^*$.  We construct an injection $\varphi : Y(n+\hat{n}) \to X(n)$.  Note that when $n$ is a multiple of $a$, the identity map gives such an injection.  Suppose $n\not\equiv 0 \pmod{a}$, and $\lambda\in Y(n+\hat{n})$ is of the form $n+\hat{n} = \lambda_1 + \cdots + \lambda_k$, where the parts are listed in increasing order.  Define $\varphi(\lambda)$ by
\[ 
\varphi(\lambda) =
\begin{cases}
(a, n-a) &\text{ if }k=1 \\
(n) &\text{ if } \lambda_1=a, k=2 \\ 
(\lambda_2, ..., \lambda_k+a-\hat{n}) &\text{ if } \lambda_1=a, k\geq 3 \\
(\lambda_1-\hat{n}, \lambda_2, ..., \lambda_k) &\text{ if }  \lambda_1 > a, k\geq 2.
\end{cases}
\]
Since $n\geq d+2a$, it follows that $\varphi(\lambda) \in X(n)$.  Moreover, observe that the first case above is the only case where the smallest part of $\varphi(\lambda)$ is $a$.  Thus, the only way for two partitions $\lambda \neq \lambda' \in Y(n+\hat{n})$ to map to the same partition under $\varphi$ would be through the third and fourth cases above.  However, since $\hat{n}\neq 0$, the last part arising from the third case is not a multiple of $a$ and thus cannot equal the any last part arising from the fourth case.  So $\varphi$ is an injection and thus $\qda \geq \q{d+\hat{d}}{a}{n+\hat{n}}^*$ as desired.  
\end{proof}

\begin{lemma}\label{lem:Qidentitylemma}
Let $a, d, n \geq 1$ such that $a$ divides $d + 3$.  Then, 
\begin{align*}
\Qdash{d}{a}{an} &= \Qdash{\frac{d+3}{a}-3}{1}{n},\\
\Qdashdash{d}{a}{an} &= \Qdashdash{\frac{d+3}{a}-3}{1}{n}.
\end{align*}

\end{lemma}
\begin{proof}
The result follows directly from the natural bijection between the relevant sets of partitions obtained by dividing (or multiplying) the parts by $a$.
\end{proof}

\section{A modification of Alder's conjecture}\label{sec:modified alders section}

In this section we utilize work of Andrews \cite{andrews_partition_1971} and Yee \cite{yee_alders_2008} to prove two modifications of \eqref{eq:alderthm}.  The first, stated below, will be used to prove Theorem \ref{thm:final kp thm} and later Proposition \ref{prop:schurprop}.

\begin{proposition}\label{prop:k-2andmlemma}
If $d=15$ or $d \geq 31$, then for $n \geq 1$, 
\[
\q{d}{1}{n} \geq \Qdash{d-2}{1}{n}.
\]
\end{proposition}

\begin{proof} 
Proposition \ref{prop:k-2andmlemma} will be proved in three cases\footnote{See Section \ref{sec:asymptotics section} for a comment on the cases when $1 \leq d < 31$, $d\neq 15$.}, which we will prove by adapting the methods of Andrews \cite{andrews_partition_1971} and Yee \cite{yee_alders_2008}.

\underline{Case I:} 
Let $d \geq 31$, $d \neq 2^{r} - 1$, where $r$ is defined in \eqref{eq:r def} with $a = 1$, and $n \geq 4d + 2^r$.  By work of Yee \cite[Lemmas 2.2 and 2.7]{yee_alders_2008} it follows that   
\begin{equation}\label{eq:yeepart}
\q{d}{1}{n} \geq \mathcal{G}_d^{(1)}(n), 
\end{equation}
where 
\begin{equation}\label{eq:Gdef}
g_d^{(1)}(q) = \sum_{k\geq 0} \mathcal{G}_d^{(1)}(n) q^n =  \frac{\pochneg{d+2^{r-1}}{2d}} {(q;q^{2d})_\infty\poch{d+2}{2d}\cdots \poch{d+2^{r-2}}{2d}}.
\end{equation}
From \eqref{eq:Gdef}, we see that $\mathcal{G}_d^{(1)}(n)$ counts partitions of $n$ into distinct parts from the set $\{ x \in \N \mid x \equiv d+2^{r-1} \!\!\! \pmod{2d}\}$ and unrestricted parts from the set 
\[
T_{r-1,d} = \{y \in \N \mid y \equiv 1, d+2, \ldots, d+ 2^{r-2} \!\!\! \pmod{2d} \}.
\]  
Thus clearly $\mathcal{G}_d^{(1)}(n) \geq \rho(T_{r-1,d}; n)$, so we have
\begin{equation}\label{eq:Yeefirst}
\q{d}{1}{n} \geq \rho(T_{r-1,d}; n).
\end{equation}
Now define
\begin{equation}\label{eq:Sdef}
S_d := \{x \in \N \mid x \equiv \pm 1 \!\!\! \pmod{d+1}\} \backslash \{ d \},
\end{equation}
so that $\rho(S_d; n) =  Q_{d-2}^{(1, -)}(n)$.  Thus in this case the result will follow from proving that $\rho(T_{r-1,d}; n) \geq \rho(S_d; n)$ for all $n\geq 1$.  Unfortunately we can't use Corollary \ref{cor:STgeneral} here, since we are only removing one part in our definition of $S_d$.  However, we can use the same method as we used in the proof of Corollary \ref{cor:STgeneral}.

By hypothesis, $d\geq 2^r-1$.  It follows that the elements of $T_{r-1,d}$, ordered by size across rows, are as shown in Figure \ref{GeneralT_r}.

\begin{figure}[h!]\caption{Elements of $T_{r-1,d}$ ordered by size across rows}\label{GeneralT_r} 
\[
\begin{array}{|c|c|c|c|}
\hline
1 & d + 2 & \cdots & d + 2^{r-2} \\
\hline
2d + 1& 3d + 2^{t+1} & \cdots & 3d + 2^{r-2} \\ 
\hline
\vdots & \vdots & \vdots & \vdots \\
\hline
(2j-2)d + 1 & (2j-1)d + 2 & \cdots & (2j-1)d + 2^{r-2} \\ 
\hline
\vdots & \vdots & \vdots & \vdots \\
\hline
\end{array}
\]
\end{figure}

Since $d\geq 31$, we have that $r\geq 5$.  As in the proof of Corollary \ref{cor:STgeneral}, it thus suffices to show that $i$th term of $T_{4,d}$ is at most the size of the $i$th term of $S_d$.  Here, $T_{4,d}$ has elements ordered by size across rows as given in Figure \ref{SpecificT_r}.

\begin{figure}[h!]\caption{Elements of $T_{4,d}$ ordered by size across rows}\label{SpecificT_r} 
\[
\begin{array}{|c|c|c|c|}
\hline
1 & d + 2 & d + 4  & d + 8 \\
\hline
2d + 1 & 3d + 2  & 3d+ 4 & 3d + 8\\ 
\hline
\vdots & \vdots & \vdots & \vdots \\
\hline
(2j-2)d + 1 & (2j-1)d + 2  & (2j-1)d + 4 & (2j-1)d + 8 \\ 
\hline
\vdots & \vdots & \vdots & \vdots \\
\hline
\end{array}
\]
\end{figure}

Since $d\geq 31$ it follows that the elements of $S_d$ are given ordered by size across rows in Figure \ref{S_delements}.

\begin{figure}[h!]\caption{Elements of $S_d$ ordered by size across rows}\label{S_delements} 
\[
\begin{array}{|c|c|c|c|}
\hline
1  & d+2 & 2d+1 & 2d+3 \\
\hline
3d+2 & 3d+4 & 4d+3 & 4d+5 \\ 
\hline
\vdots & \vdots & \vdots & \vdots \\
\hline
(2j-1)d + (2j-2) & (2j-1)d + 2j & (2j)d + (2j-1) & (2j)d + (2j+1)  \\ 
\hline
\vdots & \vdots & \vdots & \vdots \\
\hline
\end{array}
\]
\end{figure}

Let $x_i$ and $y_i$ be the $i$th smallest elements of $S_d$ and $T_{4,d}$, respectively.  We now prove that $x_i \geq y_i$ by considering the $j$th rows of the tables in Figures \ref{S_delements} and \ref{SpecificT_r}.  Each of the four cases follows in a straightforward manner from the fact that $d \geq 31$. Thus a direct application of Theorem \ref{thm:andrews ST} concludes the proof of this case.  

\underline{Case II:}
Let $d = 2^s-1$ for $s \geq 4$, and $n\geq 1$.  By Andrews \cite[Theorem 1]{andrews_partition_1971}, it is known that 
\[
q_d^{(1)}(n) \geq \mathcal{L}_d(n),
\]
where $\mathcal{L}_d(n)$ counts the number of partitions of $n$ into distinct parts $\equiv 1,2, \ldots, 2^{s-1} \pmod{d}$. Since
\begin{align*}
\sum_{n = 0}^{\infty} \mathcal{L}_d(n) q^{n} 
&= (-q;q^d)_\infty \pochneg{2}{d} \cdots \pochneg{2^{s - 1}}{d}\\
&=\frac{(q^2;q^{2d})_\infty}{ (q;q^{2d})_\infty (q^{d+1};q^{2d})_\infty} \cdots \frac{(q^{2^s};q^{2d})_\infty}{ (q^{2^{s-1}};q^{2d})_\infty (q^{d+2^{s-1}};q^{2d})_\infty} \\
&= \frac{1}{(q;q^{2d})_\infty (q^{d+2};q^{2d})_\infty \cdots (q^{d+2^{s-1}};q^{2d})_\infty},
\end{align*}
we see that in fact for all $n$,
\begin{equation}\label{eq:Andrewsfirst}
q_d^{(1)}(n) \geq  \mathcal{L}_d(n)  = \rho(T_{s,d}; n), 
\end{equation}
where $T_{s,d} := \{y \in \N \mid y \equiv 1, d+2, \cdots, d+2^{s-1} \pmod{2d} \}$.  Let $S_d$ be defined as before in \eqref{eq:Sdef} so that $\rho(S_d; n) =  Q_{d-2}^{(1, -)}(n)$.  Then to finish the proof it suffices to show $\rho(T_{s,d}; n) \geq \rho(S_d; n)$.  This follows from the previous case, since the sets $T_{4,d}$, $S_d$ are exactly those appearing in that setting. 
    
\underline{Case III:}
Let $d \geq 15$ and $1\leq n < 4d + 2^{r}$, where $r$ is defined in \eqref{eq:r def} with $a = 1$.  It thus follows that $n \leq 5d$.  Again we let $S_d$ be defined as in \eqref{eq:Sdef} so that $\rho(S_d; n) =  Q_{d-2}^{(1, -)}(n)$.  Here we will prove the lemma in cases depending on the size of $n$ in relation to the sizes of available parts in $S_d$. 

We write $\lambda^x$ to denote that the part $\lambda$ appears $x$ times as a part in a partition.  Moreover, observe that $\q{d}{1}{n}$ is a (weakly) increasing function of $n$, since adding $1$ to the largest part of any partition of $n$ counted by $\q{d}{1}{n}$ yields a partition of $n+1$ counted by $\q{d}{1}{n+1}$.  Similarly $Q_{d-2}^{(1,-)}(n)$ is a (weakly) increasing function of $n$, since adding one part of size $1$ to any partition of $n$ counted by $Q_{d-2}^{(1,-)}(n)$ will yield a partition of $n+1$ counted by $Q_{d-2}^{(1,-)}(n+1)$.

First consider the case when $1\leq n \leq d+1$.  The only available part in $S_d$ in this case is $1$, so $Q_{d-2}^{(1,-)}(n)=1$.  Also $q_d^{(1)}(n)= 1$, since it only counts the partition $(n)$.
    
When $d+2 \leq n \leq d + 3$, it follows that $\Qdash{d-2}{1}{n} = 2$, counting the partitions $(n)$ and $(1^n)$. Likewise $\q{d}{1}{n} = 2$, as it will count the partitions $(n)$ and $(1, n-1)$. 
    
Now consider $d+4 \leq n \leq 2d +2$. Here $\Qdash{d-2}{1}{n} \in \{2, 3\}$, counting $(1^n)$, $(1^{n - d -2}, (d + 2))$, and possibly $(1^{n - 2d -1},(2d - 1))$.  However, $\q{d}{1}{n} \geq \q{d}{1}{d+4} = 3$, since $\q{d}{1}{d+4}$ counts the partitions $(d+4)$, $(1,(d + 3))$, and $(2, (d + 2))$. 
     
We now consider $2d + 3 \leq n \leq 4d + 1$.  Here $\Qdash{d-2}{1}{n} \leq \Qdash{d-2}{1}{4d+1}=10$, which we calculate by observing that there is exactly one partition of $4d+1$ with parts in $S_d$ and largest part equal to each of $3d+4$, $3d+2$, and $1$, there are two each with largest part $2d+3$ and $2d+1$, and three with largest part $d+2$.  Thus since $\q{d}{1}{n} \geq \q{d}{1}{2d+3}$, it suffices to show that $\q{d}{1}{2d+3} \geq 10$.  We can see this by observing that since $d\geq 15$, the partitions $(2d+3)$, $(1, 2d+2)$, $(2, 2d+1)$, $\ldots$, $(9, 2d-6)$ are all counted by $\q{d}{1}{2d+3}$. 

Lastly, we consider $4d+2 \leq n \leq 5d$.  Here $\Qdash{d-2}{1}{n} \leq \Qdash{d-2}{1}{5d}=20$, which we calculate as above in terms of largest parts.  There is one partition of $5d$ with parts in $S_d$ for each largest part in $\{4d+5, 4d+3, 1\}$, two for each largest part in $\{3d+4, 3d+2\}$, five with largest part $2d+3$, and four for each largest part in $\{2d+1, d+2\}$.  Here $\q{d}{1}{n} \geq \q{d}{1}{4d+2}$, so it suffices to show that $\q{d}{1}{4d+2} \geq 20$.  We can see this by observing that since $d\geq 15$, the partitions $(4d+2)$, $(1, 4d+1)$, $(2, 4d)$, $\ldots$, $(19, 4d-17)$ are all counted by $\q{d}{1}{4d+2}$. 
\end{proof} 

To prove Theorem \ref{thm:stringarbitrary} we will make use of the following analogous result. 

\begin{proposition}\label{prop:k-3andmdashdashlemma}
If $d=15$ or $d \geq 31$, then for $n \geq 1$, 
\[
\q{d}{1}{n} \geq \Qdashdash{d-3}{1}{n}.
\]
\end{proposition}

\begin{proof} 

As in the proof of Proposition \ref{prop:k-2andmlemma} we prove the result in three cases.

\underline{Case I:}
Let $d \geq 31$, $d \neq 2^{r} - 1$ with $r$ defined in \eqref{eq:r def} with $a = 1$, and $n \geq 4d + 2^r$. It follows that $r\geq 5$.  Recall from \eqref{eq:Yeefirst} that 
\[
\q{d}{1}{n} \geq \rho(T_{r-1,d}; n),
\]
where
\[
T_{r-1,d} = \{y \in \N \mid y \equiv 1, d+2, \ldots, d+ 2^{r-2} \!\!\! \pmod{2d} \}.
\]
Furthermore, let
\begin{equation}\label{eq:Sdef2}
S_{1,d-3} := \{x \in \N \mid x \equiv \pm 1 \!\!\! \pmod{d}\} \backslash \{1, d-1 \},
\end{equation}
so that $\rho(S_{1,d-3}; n) =  Q_{d-3}^{(1, -)}(n)$.  Thus it suffices to show that $\rho(T_{r-1,d}; n) \geq \rho(S_{1,d-3}; n)$, which follows directly from Corollary \ref{cor:STgeneral} by setting $t=0$.

\underline{Case II:}
Let $d = 2^{s}- 1$ for $s \geq 4$, and $n\geq 1$.  Recall from \eqref{eq:Andrewsfirst} that
\[
q_d^{(1)}(n) \geq \rho(T_{s,d}; n), 
\]
where $T_{s,d} := \{y \in \N \mid y \equiv 1, d+2, \cdots, d+2^{s-1} \pmod{2d} \}$.  Let $S_{1,d-3}$ be defined as before in \eqref{eq:Sdef2} so that $\rho(S_{1,d-3}; n) =  Q_{d-3}^{(1, -)}(n)$.  Thus it suffices to show that $\rho(T_{s,d}; n) \geq \rho(S_{1,d-3}; n)$, which again follows directly from Corollary \ref{cor:STgeneral} with $t=0$.

\underline{Case III:} 
Let $d \geq 5$ and $1\leq n < 4d + 2^{r}$, where $r$ is defined in \eqref{eq:r def} with $a = 1$.  It thus follows that $n \leq 5d$.  Again we let $S_{1,d-3}$ be defined as in \eqref{eq:Sdef2} so that $\rho(S_{1,d-3}; n) =  Q_{d-3}^{(1, -)}(n)$.  As in the proof of Proposition \ref{prop:k-2andmlemma}, we prove this lemma in cases depending on the size of $n$.  While $\q{d}{1}{n}$ is a (weakly) increasing function of $n$, here $Q_{d-3}^{(1,-,-)}(n)$ is not.  However, we can calculate all the partitions counted by $Q_{d-3}^{(1,-,-)}(n)$ in the range $1\leq n \leq 5d$ in a straightforward manner grouping by largest part.  For each possible largest part, the values which can be partitioned into parts from $S_{1,d-3}$ are distinct, and are listed in Figure \ref{fig:counting}.

\begin{figure}[h!]\caption{Values partitioned into parts from $S_{1,d-3}$ for $1\leq n \leq 5d$}\label{fig:counting} 
\[
\begin{array}{|c|c|}
\hline
\text{largest part} & \text{distinct values of }n \\
\hline
4d+1 & 4d+1 \\ 
\hline
4d-1 & 4d-1, 5d \\
\hline
3d+1 & 3d+1, 4d+2, 5d \\ 
\hline
3d-1 & 3d-1, 4d, 5d-2, 5d \\
\hline
2d+1 & 2d+1, 3d+2, 4d, 4d+2, 4d+3 \\ 
\hline
2d-1 & 2d-1, 3d, 4d+1, 4d-2 \\
\hline
d+1 & d+1, 2d+2, 3d+3, 4d+4 \\ 
\hline
\end{array}
\]
\end{figure}

From Figure \ref{fig:counting}, we see that $Q_{d-3}^{(1,-,-)}(n)\leq 3$ for $1\leq n \leq 5d$. When $d+4 \leq n \leq 5d$, $\q{d}{1}{n} \geq \q{d}{1}{d+4} = 3$, since $\q{d}{1}{d+4}$ counts the partitions $(d+4)$, $(1,(d + 3))$, and $(2, (d + 2))$. Furthermore when $1 \leq n \leq d+3$, $\q{d}{1}{n} \geq 1$ and $Q_{d-3}^{(1,-,-)}(n)\leq 1$. Thus the lemma is proved.
\end{proof}

\section{A proof of Kang and Park's conjecture for $d\geq 62$}\label{sec:kang and park conjecture section}

In this section we first prove Theorem \ref{thm:final kp thm}, utilizing Lemmas \ref{lem:STlemma}, \ref{lem:qstarlemma}, and \ref{lem:Qidentitylemma}, as well as Proposition \ref{prop:k-2andmlemma}.  We then provide a partial result toward Conjecture \ref{conj:schurconj}.

\begin{proof}[Proof of Theorem \ref{thm:final kp thm}]
Let $d \geq 62$ and $n \geq 1$.  We observe that if $n \leq d+3$, then $2$ is the only possible part in a partition counted by $\Qdash{d}{2}{n}$, and thus $\Qdash{d}{2}{n}\leq 1 = q_d^{(2)}(n)$.

Now let $n\geq d+4$.  We approach the proof by cases, considering the parities of $d$ and $n$.

First, we suppose $d$ is odd.  If $n$ is also odd, then $\Qdash{d}{2}{n}=0$, so $\Delta_d^{(2)}(n)\geq 0$ trivially in this case.  If $n=2n'$ is even, then 
\[
\q{d}{2}{2n'} \geq \q{\frac{d+1}{2}}{1}{n'} \geq \Qdash{\frac{d-3}{2}}{1}{n'} = \Qdash{d}{2}{2n'},
\]
where the first inequality follows directly from Lemma \ref{lem:qstarlemma}, the second from Proposition \ref{prop:k-2andmlemma}, and the third from Lemma \ref{lem:Qidentitylemma}. 

We now suppose $d$ is even.  If $n=2n'$ is also even, then we first observe that 
\[
\q{d}{2}{2n'} \geq \q{\frac{d}{2}}{1}{n'} \geq \Qdash{\frac{d}{2} - 2}{1}{n'} = \Qdash{d-1}{2}{2n'}, 
\]
where again the first inequality follows directly from Lemma \ref{lem:qstarlemma}, the second from Proposition \ref{prop:k-2andmlemma}, and the third from Lemma \ref{lem:Qidentitylemma}.  It remains to show $\Qdash{d-1}{2}{2n'} \geq \Qdash{d}{2}{2n'}$, which we do with an application of Lemma \ref{lem:STlemma}.  Let
\begin{align}\label{eq:STdefs}
S &= \{ x\in \N \mid x \equiv \pm 2 \!\!\! \pmod{d+3} \} \backslash \{d+1\}, \\
T &= \{ y\in \N \mid y \equiv \pm 2 \!\!\! \pmod{d+2} \} \backslash \{d\}, \nonumber
\end{align}
so that $\rho(S;2n') = \Qdash{d}{2}{2n'}$ and $\rho(T;2n') =\Qdash{d-1}{2}{2n'}$.  Ordering the elements of $S$ and $T$ by size, we have 
\begin{align*}
S &= \{ 2, d+5, 2d+4, 2d+8,  \ldots, jd+3j-2 , jd+ 3j+2,  \ldots \}, \\
T &= \{2, d+4, 2d+2, 2d+6, \ldots, jd+2j-2 , jd+ 2j+2,  \ldots \},
\end{align*}
so it is easy to see that the hypotheses of Lemma \ref{lem:STlemma} are satisfied with $m=2$.  Thus it follows that $\Qdash{d-1}{2}{2n'} \geq \Qdash{d}{2}{2n'}$ as desired. 

In the case when $n=2n'-1$ is odd, while $d$ is even, we observe that 
\[
\q{d}{2}{2n'-1} \geq \q{\frac{d}{2}}{1}{n'} \geq \Qdash{\frac{d}{2} - 2}{1}{n'} = \Qdash{d-1}{2}{2n'},
\]
where the first inequality follows directly from Lemma \ref{lem:qstarlemma}, the second from Proposition \ref{prop:k-2andmlemma}, and the third from Lemma \ref{lem:Qidentitylemma}.  It remains to show that $\Qdash{d-1}{2}{2n'} \geq \Qdash{d}{2}{2n'-1}$, which we do by constructing an injection from the set $X$ of partitions counted by $\Qdash{d}{2}{2n'-1}$ to the set $Y$ of partitions counted by $\Qdash{d-1}{2}{2n'}$.  As in the previous case, let $S=\{x_i\}_{i=1}^\infty$, $T=\{y_i\}_{i=1}^\infty$ be increasing, where $S,T$ are defined in \eqref{eq:STdefs}. Suppose $\lambda\in X$ has parts $x_{i_j}\in S$, $j\geq 1$.  Then $2n'-1 = \sum_{j\geq 1} x_{i_j}$ is odd, so since each $y_i$ is even it follows that $\sum_{j\geq 1} (x_{i_j}-y_{i_j})$ is also odd.  Let $\beta$ be the positive integer with $\sum_{j\geq 1} (x_{i_j}-y_{i_j})=2\beta-1$.  We then define $\varphi : X \rightarrow Y$ by defining $\varphi(\lambda)$ to be the partition with corresponding parts $y_{i_j}\in T$, plus $\beta$ additional parts of size $2$.  Then $\varphi(\lambda)$ is a partition of 
\[
\sum_{j\geq 1} y_{i_j} + 2\beta = \sum_{j\geq 1} x_{i_j} + 1 = 2n',
\]
so $\varphi(\lambda)\in Y$.  Moreover, it is easy to see that $\varphi$ is an injection, since for any $\lambda\neq \lambda'\in X$, there must exist a part $x_i>2$ that occurs in $\lambda$ a different number of times than it does in $\lambda'$.  But then $y_i$ must occur in $\varphi(\lambda)$ a different number of times than it does in $\varphi(\lambda')$ which implies $\varphi(\lambda) \neq \varphi(\lambda')$.   
\end{proof}

We conclude this section with a partial result toward Conjecture \ref{conj:schurconj}.

\begin{proposition}\label{prop:schurprop}
For $d\geq 31$ and $n \geq d+6$, 
\[ 
\Ddasha{3d}{3}{n} \geq 0. 
\]
\end{proposition}

\begin{proof}
First observe that if $n$ is not divisible by $3$, then $\Qdash{3d}{3}{n} = 0$ so $\Ddasha{3d}{3}{n}$ is trivially nonnegative.  If $n=3n'$, then 
\[
\q{3d}{3}{3n'} \geq \q{d}{1}{n'} \geq \Qdash{d-2}{1}{n'} =\Qdash{3d}{3}{3n'},
\]
where the first inequality follows directly from Lemma \ref{lem:qstarlemma}, the second from Proposition \ref{prop:k-2andmlemma}, and the third from Lemma \ref{lem:Qidentitylemma}. 
\end{proof}

\section{A generalization of Kang and Park's conjecture}\label{sec:generalized kang and park conjecture section}

In this section we use three different methods to prove partial results toward Conjecture \ref{conj:Generalized KP Conjecture}. We begin by adapting methods developed in the proof of Theorem \ref{thm:final kp thm} to prove Theorem \ref{thm:stringarbitrary}. We then use methods of Andrews \cite{andrews_partition_1971} to prove Theorem \ref{thm:andrewsarbitrary} and methods of Yee \cite{yee_alders_2008} to prove Theorem \ref{thm:yeearbitrary}. 

\begin{proof}[Proof of Theorem \ref{thm:stringarbitrary}.]
Let $a \geq 3$, and $d\geq 31a - 3$, such that $a$ divides $d+3$.  We observe that if $n \leq d+2a-1$, then $d+3+a$ is the only possible part in a partition counted by $\Qdashdash{d}{a}{n}$, and thus $\Qdashdash{d}{a}{n}\leq 1 \leq q_d^{(a)}(n)$.  Now let $n\geq d+2a$.  We first observe for $n$ that are not divisible by $a$, $\Qdashdash{d}{a}{n} = 0$, and thus $\Ddashdasha{d}{a}{n} \geq 0$ is trivially true.  It thus suffices to consider $n=an'$, for $n' \geq 1$.  When $a\geq 4$, then $\lceil \frac{d}{a} \rceil = \frac{d+3}{a}$.  Thus,
\[
\q{d}{a}{an'} \geq \q{\frac{d+3}{a}}{1}{n'} \geq \Qdashdash{\frac{d+3}{a}-3}{1}{n'} =\Qdashdash{d}{a}{an'},
\]
where the first inequality follows directly from Lemma \ref{lem:qstarlemma}, the second from Proposition \ref{prop:k-3andmdashdashlemma}, and the third from Lemma \ref{lem:Qidentitylemma}. 

If $a=3$, then $d$ is divisible by $3$, so we have
\[
\q{d}{3}{3n'} \geq \q{\frac{d}{3}}{1}{n'} \geq \q{\frac{d}{3}+1}{1}{n'} \geq  \Qdashdash{\frac{d}{3}-2}{1}{n'}=\Qdashdash{d}{3}{3n'},
\] 
using Lemma \ref{lem:qstarlemma}, Proposition \ref{prop:k-3andmdashdashlemma}, Lemma \ref{lem:Qidentitylemma}, and the definition of $\q{d}{a}{n}$.
\end{proof}

In order to prove Theorems \ref{thm:andrewsarbitrary} and \ref{thm:yeearbitrary} we require a useful theorem of Andrews \cite{andrews_partition_1971}, stated below\footnote{The statement of Theorem \ref{thm:andrewsDandEsets} is adjusted to our specific context.}, and the following notation.  For integers $0\leq k \leq \ell$ and $d\geq 2^\ell - 2^k$, let
\begin{align}
A_{d,k,\ell} &= \{x \in \N \mid x \equiv 2^i \!\!\! \pmod{d}, k \leq i \leq \ell-1\}, \nonumber \\
A_{d,k,\ell}' &= \{y \in \N \mid y \equiv 2^k i  \!\!\! \pmod{d}, 1 \leq i \leq 2^{\ell-k} - 1\}. \label{A'def}
\end{align}
Let $\beta_d(x)$ be the least positive residue of $x$ modulo $d$, let $b(x)$ be the number of (nonzero) terms appearing in the binary representation of $x$, and let $\nu(x)$ be the least $2^i$ in this representation.

\begin{theorem}[Andrews \cite{andrews_general_1969}, 1969] \label{thm:andrewsDandEsets}
Using the notation above, let $D(A_{d,k,\ell};n)$ denote the number of partitions of $n$ into distinct parts taken from $A_{d,k,\ell}$, and let $E(A_{d,k,\ell}';n)$ be the number of partitions of $n$ into parts taken from $A_{d,k,\ell}'$ of the form $n = \la_1 + \la_2 + \cdots + \la_s$, such that 
\[
\la_{i + 1} - \la_{i} \geq d \cdot b(\beta_d(\la_i)) + \nu(\beta_d(\la_i))- \beta_d(\la_i). 
\]
Then $D(A_{d,k,\ell};n)=E(A_{d,k,\ell}';n)$.
\end{theorem}

\begin{proof}[Proof of Theorem \ref{thm:andrewsarbitrary}]

Let $1\leq a \leq d+2$, where $d = 2^s-2^{t}$ for some $t \geq \log_2(a)$ and $s \geq t + 4$.  Consider $A_{d,t,s}'$ as defined in \eqref{A'def}.  Then any partition counted by $E(A_{d,t,s}';2^t n)$ has parts at least $2^t \geq a$, and the difference between parts satisfies
\begin{equation}\label{eq:diffcond}
\la_{i + 1} - \la_{i} \geq d \cdot b(\beta_d(\la_i)) + \nu(\beta_d(\la_i))- \beta_d(\la_i). 
\end{equation}
Let $\lambda_i$ be a part of such a partition counted by $E(A_{d,t,s}';2^t n)$.  If $\beta_d(\la_i) = 2^j$ for some $t\leq j \leq s-1$, then $\nu(2^j)=\beta(2^j)$, so \eqref{eq:diffcond} implies that $\la_{i + 1} - \la_{i} \geq d$.  If $\beta_d(\la_i) \neq 2^j$ for some $t\leq j \leq s-1$, then $b(\beta_d(\la_i))\geq 2$.  Since $\beta_d(\la_i) \leq 2^s-2^t=d$, \eqref{eq:diffcond} again implies that $\la_{i + 1} - \la_{i} \geq d$.  Thus we can conclude by Theorem \ref{thm:andrewsDandEsets} that 
\begin{equation}\label{eq:andrewsineq1}
q_d^{(a)}(2^tn) \geq E(A_{d,t,s}';2^t n) = D(A_{d,t,s};2^t n),
\end{equation}
where $D(A_{d,t,s};2^t n)$ counts the number of partitions of $n$ into distinct parts from $A_{d,t,s}$, as defined in \eqref{A'def}. Since
\begin{align*}
\sum_{n = 0}^{\infty} D(A_{d,t,s};n) q^{n} 
&= \pochneg{2^{t}}{d} \pochneg{2^{t+1}}{d} \cdots \pochneg{2^{s - 1}}{d}\\
&=\frac{\poch{2^{t+1}}{2d}}{\poch{2^{t}}{2d} \poch{d + 2^{t}}{2d}} \cdots \frac{\poch{2^{s}}{2d}}{\poch{2^{s - 1}}{2d}\poch{d + 2^{s - 1}}{2d}}\\
&= \frac{1}{\poch{2^{t}}{2d}\poch{d + 2^{t + 1}}{2d} \cdots \poch{d + 2^{s-1}}{2d}},
\end{align*}
we see that in fact for all $n$,
\begin{equation}\label{eq:andrewsineq2}
D(A_{d,t,s};n) = \rho(T_{t,s,d}; n), 
\end{equation}
where $T_{t,s,d} := \{y \in \N \mid y \equiv 2^{t}, d+2^{t+1}, \cdots, d+2^{s-1} \pmod{2d} \}$.   

Now define
\[
S_{a,d} := \{x \in \N \mid x \equiv \pm a \!\!\! \pmod{d+3}\} \backslash \{ a, d+3-a \},
\]
so that $\rho(S_{a,d}; n) =  \Qdadashdash$.  From \eqref{eq:andrewsineq1} and \eqref{eq:andrewsineq2} we see that to finish the proof it suffices to show $\rho(T_{t,s,d}; 2^t n) \geq \rho(S_{a,d}; 2^t n)$.  This follows directly from Corollary \ref{cor:STgeneral} by setting $b=d$, which completes the proof.
\end{proof} 

\subsection{The proof of Theorem \ref{thm:yeearbitrary}}

In this subsection we modify a method of Yee \cite{yee_alders_2008} to prove Theorem \ref{thm:yeearbitrary}.  We do this by proving, under the hypotheses of the theorem, the following series of inequalities 
\begin{equation*}
q_{d}^{(a)} (2^{a-1}n) 
\geq \mathcal{L}_d^{(a)}(2^{a-1}n) + \mathcal{L}_d^{(a)}(2^{a-1}n-2^r)
\geq \mathcal{K}_d^{(a)}(2^{a-1}n)
\geq \mathcal{G}_d^{(a)}(2^{a-1}n)
\geq Q_d^{(a,-,-)}(2^{a-1}n),
\end{equation*}
where
\begin{align}\label{fkg_defs}
\fda &:= \sum_{n = 0}^{\infty} \Lda q^n := \pochneg{2^{a-1}}{d}\pochneg{2^{a}}{d} \cdots \pochneg{2^{r-1}}{d}, \nonumber \\
\kda &:= \sum_{n = 0}^{\infty} \Kda q^n := \frac{1-q^{d+2^{a-1}}}{1-q^{2^{a-1}}} \pochneg{d+2^{a - 1}}{d} \pochneg{d+2^{a}}{d}\cdots \pochneg{d+2^{r-1}}{d}, \\
\gda &:= \sum_{n = 0}^{\infty} \Gda q^n := \frac{\pochneg{d+2^{r-1}}{2d}} {\poch{2^{a-1}}{2d}\poch{d+2^{a}}{2d}\cdots \poch{d+2^{r-2}}{2d}}. \nonumber 
\end{align}

Since the first and third inequality above require a substantial combinatorial argument, we first address those as separate lemmas.  Our first lemma is as follows.

\begin{lemma}\label{yee2.7analog}
Let $a\geq 3$ and $d \geq 2^{a+2}$, where $d\not= 2^r - 2^{a-1}$ with $r$ defined by \eqref{eq:r def}.  Then for all $n\geq 2d+2^r + 2^{a-1}-1$, 
\[
\q{d}{a}{n} \geq \L{a}{n} + \L{a}{n - 2^r}. 
\]
\end{lemma}

\begin{proof}
By \eqref{fkg_defs}, we see that $\L{a}{n}$ counts the number of partitions into distinct parts coming from the set $A_{d,a-1,r}$ as defined in \eqref{A'def}, so by Theorem \ref{thm:andrewsDandEsets} it follows that
\begin{equation}\label{eqnL_E}
\L{a}{n} = D(A_{d,a-1,r};n)=E(A_{d,a-1,r}';n).
\end{equation}
First we observe that by \eqref{eqnL_E} it is equivalent to show that 
\[
\q{d}{a}{n} \geq E(A_{d,a-1,r}';n) + E(A_{d,a-1,r}';n - 2^r).
\]
Recall from the statement of Theorem \ref{thm:andrewsDandEsets} that $E(A_{d,a-1,r}';m)$ counts the number of partitions of \(n\) into parts taken from $A_{d,a-1,r}'$ of the form \(n = \la_1 + \la_2 + \cdots + \la_s\) such that 
\begin{equation}\label{diffcond}
\la_{i + 1} - \la_{i} \geq d \cdot b(\beta_d(\la_i)) + \nu(\beta_d(\la_i))- \beta_d(\la_i). 
\end{equation}
Define $Y_d^{(a)}(n)$ to be the set of partitions counted by $E(A_{d,a-1,r}';n)$. Additionally, let $X_d^{(a)}(n)$ denote the set of partitions counted by $q_d^{(a)}(n)$, i.e., the partitions of $n$ with parts $\geq a$ having difference at least $d$. By \eqref{A'def}, partitions in $Y_d^{(a)}(n)$ have parts $\leq a$.  Moreover, \eqref{diffcond} shows that the difference between parts must be at least $d$.  This can be seen from the facts that $\nu(\beta_d(\la_i))= \beta_d(\la_i)$ when $\la_i$ is a power of $2$, and $\beta_d(\la_i)\leq 2^{r} - 2^{a-1} < d$, when $\la_i$ is not a power of $2$.  Thus we have $Y_d^{(a)}(n)\subseteq X_d^{(a)}(n)$, and so to prove the lemma it suffices to show that 
\begin{equation}\label{setgoal}
|Y_d^{(a)}(n-2^r)| \leq |X_d^{(a)}(n)\backslash Y_d^{(a)}(n)|. 
\end{equation}

In a 2008 paper, Yee \cite{yee_alders_2008} constructed an injective map
\begin{equation}\label{eq:psi}
\psi: Y_d^{(1)}(n-2^r)\rightarrow X_d^{(1)}(n)\backslash Y_d^{(1)}(n),
\end{equation}
which proves \eqref{setgoal} for $a=1$.  We note that although Yee \cite{yee_alders_2008} assumes $n \geq 4d+2^r$ in the proof, it is only necessary that $n \geq 2d+2^r$.  By the definitions of $X_d^{(a)}(n)$ and $Y_d^{(a)}(n)$, it is easy to see that $Y_d^{(a)}(n) \subseteq Y_d^{(1)}(n)$ and $X_d^{(a)}(n) \subseteq X_d^{(1)}(n)$.  Thus to finish the proof it would suffice to show that the restriction of Yee's map to $Y_d^{(a)}(n-2^r)$,
\[
\Psi: Y_d^{(a)}(n-2^r)\rightarrow X_d^{(1)}(n)\backslash Y_d^{(1)}(n) \subseteq X_d^{(1)}(n)\backslash Y_d^{(a)}(n),
\]
has image contained in $X_d^{(a)}(n)$.  Since we already know the image is in $X_d^{(1)}(n)$, we would only need to show that parts of partitions in the image are $\geq a$.

Yee's map $\psi$ in \eqref{eq:psi} is defined piecewise on four sets (in the case $a=1$), defined by 
\begin{align*}
Z^{(a)} &:= \{\lambda \in Y_d^{(a)}(n-2^r) \mid \lambda_{i+1}-\lambda_i \geq 2d, \mbox{ and } \beta_d(\lambda_i) + 2^r \leq d  \mbox{ for some } 1\leq i \leq s\}  \\
V_1^{(a)} & := \{ \lambda \in Y_d^{(a)}(n-2^r) \backslash Z^{(a)} \mid \lambda_1 \geq 2d + 2^{a-1} - 1\}, \\
V_2^{(a)} & := \{ \lambda \in Y_d^{(a)}(n-2^r) \backslash Z^{(a)} \mid d \leq \lambda_1 < 2d + 2^{a-1} - 1\}, \\
V_3^{(a)} & := \{ \lambda \in Y_d^{(a)}(n-2^r) \backslash Z^{(a)} \mid \lambda_1 < d\}, \\
\end{align*}
where $\lambda = \lambda_1 + \cdots + \lambda_s$ with parts listed in increasing order.  Observe that by definition of $Z^{(a)}$, we have $Y_d^{(a)}(n-2^r) \cap Z^{(1)} = Z^{(a)}$, so it follows that $Y_d^{(a)}(n-2^r) \backslash Z^{(a)} \subseteq Y_d^{(1)}(n-2^r) \backslash Z^{(1)}$, so the sets above are all subsets of their $a=1$ versions. Moreover the set $V_3^{(a)}$ is further partitioned into the five subsets
\begin{align*}
V_{3,1}^{(a)} & := \{ \lambda \in V_3^{(a)} \mid \lambda_2 - \beta_d(\lambda_2) \geq 6d \}, \\
V_{3,2}^{(a)} & := \{ \lambda \in V_3^{(a)} \mid \lambda_2 - \beta_d(\lambda_2) =4d \mbox{ or }5d \}, \\
V_{3,3}^{(a)} & := \{ \lambda \in V_3^{(a)} \mid \lambda_2 - \beta_d(\lambda_2) = 3d \}, \\
V_{3,4}^{(a)} & := \{ \lambda \in V_3^{(a)} \mid \lambda_2 - \beta_d(\lambda_2) = 2d \}, \\
V_{3,5}^{(a)} & := \{ \lambda \in V_3^{(a)} \mid \lambda_2 - \beta_d(\lambda_2) = d \}. \\
\end{align*}
where we note that by our hypothesis on the size of $n$, the number of parts is $s\geq 2$ for any $\lambda \in V_2^{(a)}\cup V_3^{(a)}$.  We are able to simply restrict $\psi$ to $\Psi$ in each case except $V_{3,3}^{(a)}$, where we need to make a slight modification to the map.  Thus for the remainder of the proof, we state Yee's definition of $\psi$ from \cite{yee_alders_2008} for each of the above subsets except $V_{3,3}^{(a)}$ and show that restricting to $Y_d^{(a)}(n-2^r)$ yields partitions into parts $\geq a$.  We then give our modification for the case $V_{3,3}^{(a)}$, and show not only that partitions in the image of $V_{3,3}^{(a)}$ have parts $\geq a$, but also that the difference between parts is $\geq d$, the map is an injection, and that the image of $V_{3,3}^{(a)}$ is disjoint with $Y_d^{(a)}(n)$ and the image of all of the other sets in the partition of $Y_d^{(a)}(n-2^r)$.

For $\lambda \in Y_d^{(a)}(n-2^r)$, write $\lambda = \lambda_1 + \cdots + \lambda_s$ with parts listed in increasing order.  Then Yee defines $\psi(\lambda) = \mu_1 + \cdots + \mu_t$, where again the parts are listed in increasing order.  Thus to show parts in the image are $\geq a$ it suffices to simply show $\mu_1\geq a$.

For our first case, let $\lambda \in Z^{(a)}$, and let $i$ be the least positive integer such that $\lambda_{i+1}-\lambda_i \geq 2d$, and $\beta_d(\lambda_i) + 2^r \leq d$.  We define $\Psi(\lambda)$ to have parts $\mu_j$, where 
\[
\mu_j = 
\begin{cases}
\la_j + 2^r, &\text{if } j = i, \\
\la_j, &\text{if } j \neq i.
\end{cases}
\]
By \eqref{A'def}, each $\lambda_j\geq 2^{a-1} \geq a$, so $\mu_1\geq a$ as desired.

Next, let  $\lambda \in V_1^{(a)}$, and define $\Psi(\lambda)$ to have parts $\mu_j$, where
\[
\mu_j = 
\begin{cases}
2^r, &\text{if } j = 1,\\
\la_{j - 1}, &\text{if } j \geq 2.
\end{cases}
\]
Here $\mu_1 = 2^r \geq 2^{a-1}\geq a$, as desired.

Next, let  $\lambda \in V_2^{(a)}$.  Define $i$ to be the least positive integer such that 
\[
(\lambda_{i+1}-\beta_d(\lambda_{i+1})) - (\lambda_i-\beta_d(\lambda_i)) \geq 2d,
\]
if such an integer exists, and set $i=s$ if not.  Define $\Psi(\lambda)$ by
\[
\mu_j = 
\begin{cases}
\beta_d(\lambda_1), &\text{if } j = 1,\\
\lambda_{j-1}-\beta_d(\lambda_{j-1}) + \beta_d(\lambda_j), &\text{if } 2 \leq j \leq i,\\
\lambda_i-\beta_d(\lambda_i) + 2^r, &\text{if } j = i + 1,\\
\la_{j - 1}, &\text{if } j \geq i + 2.
\end{cases}
\] 
By \eqref{A'def}, $\mu_1 = \beta_d(\lambda_1) \geq 2^{a-1} \geq a$ as desired.

Next, let $\lambda \in V_{3,1}^{(a)}$.  Define $\Psi(\lambda)$ by 
\[
\mu_j = 
\begin{cases}
2^r - 1, &\text{if } j = 1,\\
2d + \beta_d(\lambda_1), &\text{if } j = 2,\\
\lambda_2 - 2d + 1, &\text{if } j = 3,\\
\la_{j - 1}, &\text{if } j \geq 4.
\end{cases}
\]  
Since we are assuming $d\geq 2^{a+2}$, \eqref{eq:r def} yields that $r\geq a+2$, so $\mu_1 = 2^r-1 \geq a$ as desired.

Next, let $\lambda \in V_{3,2}^{(a)}$.  Define $\Psi(\lambda)$ by    
\[
\mu_j = 
\begin{cases}
\lambda_1, &\text{if } j = 1,\\
d + 2^r - 1, &\text{if } j = 2,\\
\lambda_2- d + 1, &\text{if } j = 3,\\
\la_{j - 1}, &\text{if } j \geq 4.
\end{cases}
\]  
By \eqref{A'def}, $\mu_1 = \lambda_1 \geq 2^{a-1} \geq a$ as desired.

Next, let $\lambda \in V_{3,4}^{(a)}$.  Define $\Psi(\lambda)$ by
\[
\mu_j = 
\begin{cases}
2^r - \lambda_1 - 1, &\text{if } j = 1,\\
\la_j, &\text{if } 2 \leq j \leq s-1,\\
\la_s + 2\lambda_1 + 1, &\text{if } j = s.
\end{cases}
\]  
By definition of $V_{3,4}^{(a)}$, we have $\lambda_2 - \beta_d(\lambda_2) = 2d$. Thus \eqref{diffcond} implies that $b(\lambda_1) \leq 2$.  Thus $b(2^r - \lambda_1 - 1) \geq r-2$.  However, again since we are assuming $d\geq 2^{a+2}$ we have $r\geq a+2$, so $\mu_1 = 2^r - \lambda_1 - 1 \geq b(2^r - \lambda_1 - 1) \geq a$ as desired. 
 
Next, let $\lambda \in V_{3,5}^{(a)}$.  Define $i$ to be the least positive integer such that 
\[
(\lambda_{i+1}-\beta_d(\lambda_{i+1})) - (\lambda_i-\beta_d(\lambda_i)) \geq 2d,
\]
if such an integer exists, and set $i=s$ if not. Set 
\[
x_\lambda = 
\begin{cases}
5, & \text{if } \beta_d(\lambda_{i - 1}) \not\in \{ 1,4\},\\
10, &\text{if } \beta_d(\lambda_{i - 1}) \in \{ 1,4\},
\end{cases}
\]
and define $\Psi(\lambda)$ by 
\[
\mu_j = 
\begin{cases}
\la_j, &\text{if } j \leq i - 2,\\
\la_j + x_\lambda, &\text{if } j = i - 1 \text{ or } i,\\
\la_{j}, &\text{if } i+1 \leq j \leq s-1,\\
\la_{s} + 2^r -2x_\lambda, &\text{if } j = s.
\end{cases}
\] 
By \eqref{A'def}, $\mu_1 \geq \lambda_1 \geq 2^{a-1} \geq a$, as desired.

It remains to consider the case when $\lambda \in V_{3,3}^{(a)}$ where we make a minor change to Yee's definition.  In this case we define $\Psi(\lambda)$ by 
\[
\mu_j = 
\begin{cases}
2^{a - 1} - 1, &\text{if } j = 1,\\
d + \lambda_1 + 1, &\text{if } j = 2,\\
\lambda_2- d + 2^r - 2^{a-1}, &\text{if } j = 3,\\
\la_{j - 1}, &\text{if } j \geq 4.
\end{cases}
\]  
By this construction, $\Psi$ is injective on $V_{3,3}^{(a)}$.  To see that the parts differ by at least $d$, we first observe that 
$\lambda_1 = \beta_d(\lambda_1) \geq 2^{a-1}$ by \eqref{A'def}.  Thus $\mu_2-\mu_1 \geq d$.  By definition of $V_{3,3}^{(a)}$, we have $\lambda_2 - \beta_d(\lambda_2) = 3d$.  Thus 
\[
\mu_3-\mu_2 \geq d \iff ((2^r-2^{a-1}) - \lambda_1) + (\beta_d(\lambda_2) - 1) \geq 0,
\]
which follows from the fact that $\lambda_1 = \beta_d(\lambda_1)$ and \eqref{A'def}.  Since $\lambda_3 - \lambda_2 \geq d$, \eqref{eq:r def} gives that $\mu_4-\mu_3 \geq d$.  Since $\lambda_{j+1} - \lambda_j \geq d$ for each $j$ we have directly that $\mu_j-\mu_{j-1}\geq d$ for $j\geq 5$.  By hypothesis we have $a\geq 3$, so it follows that $\mu_1 = 2^{a-1}-1 \geq a$.  It remains to show that  $\Psi(V_{3,3}^{(a)})$ does not intersect with the images of any of the other sets in the partition of $Y_d^{(a)}(n-2^r)$.     
    
Since $\mu_1 = 2^{a-1}-1$,  $\Psi(\lambda) \not \in Y_d^{(a)}(n)$, by definition of $A^{(a)'}_d$ in \eqref{A'def}.  Moreover, since 
\[
\mu_1 = 2^{a-1}-1< \beta_d(\lambda_1) < 2^r,
\]
it follows that $\Psi(\lambda) \not \in \Psi(Z^{(a)})$, $\Psi(\lambda) \not \in \Psi(V_1^{(a)})$, $\Psi(\lambda) \not \in \Psi(V_2^{(a)})$, and $\Psi(\lambda) \not \in \Psi(V_{(3,c)}^{(a)})$ for $c=1,2,5$.  To see that $\Psi(\lambda) \not \in \Psi(V_{(3,4)}^{(a)})$, we observe that the only way $\mu_1 = 2^{a-1}-1 = 2^r - \lambda_1' - 1$ is if $\lambda_1' = 2^r-2^{a-1}$ where $\lambda_1'$ is the smallest part of a partition in $V_{(3,4)}^{(a)}$.  But we saw in that case that $b(\lambda_1') \leq 2$, so either $\lambda_1' = 2^k$ or $\lambda_1' = 2^k + 2^\ell$.  As we have observed, $d\geq 2^{a+2}$ implies that $r\geq a+2$, and thus, $2^{a-1} < \lambda_1 <2^r$.  So $a-1< k < \ell < r$ (omitting $\ell$ if $b(\lambda_1')=1$), and since distinct powers of $2$ cannot sum to a power of $2$, we conclude $\lambda_1' \neq 2^r-2^{a-1}$ which finishes the proof.     
\end{proof}

We now turn our attention to the third inequality.
\begin{lemma}\label{Lemma2.2analog}
For any $a\geq 1$, and $d \geq 2^{a+1}$, where $d\not= 2^r - 2^{a-1}$ with $r$ defined by \eqref{eq:r def}, we have for $n \geq 1$,
\[
\Kda \geq \Gda.
\]
\end{lemma}

\begin{proof}
This proof is a generalization of, and follows the method of Yee \cite[Lemma 2.2]{yee_alders_2008}.  First, by our hypothesis that $d\geq 2^{a+1}$, we have by \eqref{eq:r def} that $r\geq a+1$.  Using \eqref{fkg_defs} we rewrite $\k{a}$ as
\begin{align}
\k{a} &= \frac{1-q^{d+2^{a-1}}}{1-q^{2^{a-1}}} \cdot \frac{\poch{2d+2^{a}}{2d}}{\poch{d+2^{a-1}}{d}} \cdot \frac{\poch{2d+2^{a+1}}{2d}}{\poch{d+2^{a}}{d}} \cdot \cdots \cdot \frac{\poch{2d+2^r}{2d}}{\poch{d+2^{r-1}}{d}}  \nonumber \\
&= \frac{\poch{4d+2^r}{4d} \pochneg{d+2^{r-1}}{2d}}{\poch{2^{a-1}}{2d} \poch{3d+2^{a-1}}{2d} \poch{d+2^{a}}{2d} \cdots \poch{d+2^{r-2}}{2d}}. \label{eq:k}
\end{align}
We observe that the presence of the term $(q^{4d+2^r};q^{4d})_\infty$ in the numerator of \eqref{eq:k} yields a partition theoretic interpretation of $\K{a}{n}$ as a difference of two types of partitions.  We define $S^+(n)$ to be the set containing partitions of $n$ into unrestricted parts $\equiv 2^{a-1}, d+2^{a-1}, d+2^{a}, ..., d+2^{r-2} \pmod{2d}$ excluding the part $d+2^{a-1}$,  distinct parts $\equiv d + 2^{r-1} \pmod{2d}$, and an \emph{even} number of distinct parts $\equiv 2^r \pmod{4d}$ that are at least $4d+2^r$.  Similarly, define $S^-(n)$ to be the set containing partitions of $n$ into unrestricted parts $\equiv 2^{a-1}, d+2^{a-1}, d+2^{a}, ..., d+2^{r-2} \pmod{2d}$ excluding the part $d+2^{a-1}$,  distinct parts $\equiv d + 2^{r-1} \pmod{2d}$, and an \emph{odd} number of distinct parts $\equiv 2^r \pmod{4d}$ that are at least $4d+2^r$.  Then we have that 
\[
\Kda = |S^+(n)|-|S^-(n)|.
\] 
Write $S(n) = S^+(n) \cup S^-(n)$, and define the sign of a partition $\pi \in S(n)$ by
\[
\textrm{sgn}(\pi) =
\begin{cases}
1 & \text{if } \pi \in S^+(n) \\
-1 & \text{if } \pi \in S^-(n). \\
\end{cases}
\] 
Additionally, let $T(n)$ denote the set of partitions of $n$ into unrestricted parts $\equiv 2^{a-1}, d+2^{a}, ..., d+2^{r-2} \pmod{2d}$ and distinct parts $\equiv d + 2^{r-1} \pmod{2d}$ so that
\[
\mathcal{G}_d^{(a)}(n) = |T(n)|.
\]
Thus it suffices to prove that $|T(n)| \leq |S^+(n)| - |S^-(n)|$, which since $T(n) \subseteq S^+(n)$ is disjoint from $S^-(n)$ is equivalent to showing
\[
|T(n)\cup S^-(n)| \leq |S^+(n)|.
\]
We prove this by constructing a bijective function $\varphi: S(n) \rightarrow S(n)$ such that 
\begin{equation}\label{eq:goal}
\varphi(T(n)\cup S^-(n))\subseteq S^+(n).
\end{equation}

The function $\varphi$ is constructed by an iterative process which we outline below.  Set $S_{a-2}:=\emptyset$.  For each $a-1\leq t \leq r-2$, we will define 
\[
\varphi_t : S(n)\backslash \bigcup_{j=a-1}^{t-1} S_j(n) \rightarrow S(n)\backslash \bigcup_{j=a-1}^{t-1} S_j(n),
\]
and correspondingly set
\begin{equation}\label{S_tdef}
S_t(n) := \{\pi \mid \varphi_t(\pi) \neq \pi \}.
\end{equation}
By definition, the sets $S_t(n)$, for $a-1\leq t \leq r-2$ are disjoint and together with $\bigcup_{t=a-1}^{r-2} S_t(n)$, partition $S(n)$.  We define $\varphi: S(n) \rightarrow S(n)$ by 
\[
\varphi(\pi) = 
\begin{cases}
\varphi_t(\pi) & \text{if } \pi \in S_t(n), \text{ for } a-1\leq t \leq r-2, \\
\pi & \text{otherwise}.\\
\end{cases}
\]
Since $T(n) \subseteq S^+(n)$, the proof of \eqref{eq:goal} can then be completed by showing that 
\begin{enumerate}
\item
$\varphi_t$ is a bijection for each $a-1\leq t \leq r-2$, 
\item 
If $\pi \in S_t(n)$, for $a-1\leq t \leq r-2$, then $\rm{sgn}(\varphi_t(\pi))=-\rm{sgn}(\pi)$, 
\item
$S^-(n) \subseteq \bigcup_{t=a-1}^{r-2} S_t(n)$, and
\item
$T(n) \subseteq S(n) \backslash \bigcup_{t=a-1}^{r-2} S_t(n)$.
\end{enumerate}

In order to define $\varphi_t$, we use the following notation used by Yee \cite{yee_alders_2008}.  Write $j \in \pi$ if $j$ is a part of $\pi$ and write $j^{x}$ to indicate $x$ occurrences of $j$ as a part.  Also, when the partition $\pi$ is set, let $m_j$ denote the number of times $j$ occurs as a part. Finally, let $\alpha := d + 2^{r} - 2^{a -1}$ where $r$ is defined by \eqref{eq:r def}.  
\begin{remark}\label{rem:parts}
In what follows we define $\varphi_t$ for the cases $t=a-1$, and $a \leq t \leq r-2$ separately.  In all cases, the only parts that are altered are either of the form $2^{r-j}kd+2^r$, where $a-1\leq j \leq r-2$ and $k\in \Z$, $2^{a-1}$, or $kd+2^t$, where $a-1\leq j \leq r-2$ and $k\in \Z$.  Clearly $2^{r-j}kd+2^r \equiv 2^r \pmod{4d}$.  Also, $2^{a-1} \not  \equiv 2^r \pmod{4d}$, since $2^{a-1} < 2^r < 4d$.  By \eqref{eq:r def} and our hypothesis that $d\not= 2^r - 2^{a-1}$ we have $2^r-2^j <d$, so it follows that $2^{r-j}kd+2^r \not \equiv 2^r \pmod{4d}$ as well.  
\end{remark}

We first define $\varphi_{a-1}$.  For a fixed partition $\pi \in S(n)$, let
\begin{align*} 
x  &:= \text{smallest integer } i \text{ such that } 2^{r-(a-1)}id+2^r \in \pi \\ 
y  &:= \text{smallest integer } j \text{ such that } (2^{r-(a-1)}j-1)d+2^{a-1} \in \pi \\
z  &:= \text{smallest } l > y \text{ such that } m_{ld + 2^{a - 1}} \geq 2^{r-(a-1)},
\end{align*} 
where if there is no such $x$, $y$, $z$, we set $x = \infty$, $y = \infty$, or $z = \infty$, respectively.  We then define $\varphi_{a-1}(\pi)$ by making the following substitutions among parts of $\pi$,
\[
\begin{cases}
2^{r-(a-1)}xd+2^r \to (2^{r-(a-1)}x - 1)d + 2^{a-1}, (2^{a-1})^{\frac{\alpha}{2^{a-1}}} & \text{if } x \leq y, x < \infty, \\
(2^{r-(a-1)}y - 1)d +2^{a-1},( 2^{a-1})^{\frac{\alpha}{2^{a-1}}} \to 2^{r-(a-1)}yd+2^r & \text{if }x > y, m_{2^{a-1}} \geq \frac{\alpha}{2^{a-1}},  \\
2^{r-(a-1)}xd + 2^{r} \to (xd + 2^{a-1})^{2^{r-(a-1)}}, & \text{if }x > y, x < \infty, m_{2^{a-1}} < \frac{\alpha}{2^{a-1}}, x \leq z,   \\
(zd + 2^{a-1})^{2^{r-(a-1)}} \to 2^{r-(a-1)}zd +2^{r}, &  \text{if }x > y, m_{2^{a-1}} < \frac{\alpha}{2^{a-1}}, x > z, 
\end{cases}
\]
and setting $\varphi_{a-1}(\pi) = \pi$ otherwise.  Under each condition, the number $n$ being partitioned doesn't change.  Since the conditions defining $\varphi_{a-1}$ are pairwise disjoint, and ensure the required parts exist, $\varphi_{a-1}$ is well-defined.  Moreover as the parts $\equiv 2^r \pmod{4d}$ in the definition of $S(n)$ are distinct, we can see that $\varphi_{a-1}^{2}(\pi) = \pi$, and so $\varphi_{a-1}$ is clearly a bijection.  By our discussion in Remark \ref{rem:parts} we see that in each case of the definition, a part $\equiv 2^r \pmod{4d}$ is either removed or added. Thus by \eqref{S_tdef}, if $\pi \in S_{a-1}(n)$, we have 
\[
\textrm{sgn}(\varphi_{a-1}(\pi)) = -\textrm{sgn}(\pi).
\]
Notice that if $x<\infty$, then $\pi$ must satisfy one of the four conditions in the definition of $\varphi_{a-1}$.  Thus for any $\pi \not \in S_{a-1}(n)$, $x=\infty$, i.e., $\pi$ has no part $\equiv 2^{r} \pmod{2^{r-(a-1)}d}$, however may have a part of the form $2^{r-a}kd + 2^r$ with $k$ odd.  We also observe that if $\pi \not \in S_{a-1}(n)$ and $y<\infty$, then $m_{2^{a-1}} < \frac{\alpha}{2^{a-1}}$ and $z = \infty$.  

We now iteratively define $\varphi_t$ for $a \leq t \leq r-2$.  For a fixed partition $\pi \in S(n)\backslash  \bigcup_{j=a-1}^{t-1} S_j(n)$, let
\begin{align*} 
u & := \text{smallest integer } p  \text{ such that } (2^{r-t+1}p-1)d+2^{a-1} \in \pi, \\
x & := \text{smallest odd integer } i \text{ such that } 2^{r-t}id+2^r \in \pi, \\ 
y & := \text{smallest odd integer } j \text{ such that } (2^{r-t}j-1)d+2^{a-1} \in \pi, \\
w & := \text{smallest odd integer } l \text{ such that } m_{ld+2^t} \geq 2^{r-t}, \\
z & := \text{smallest odd integer } l > y \text{ such that } m_{ld + 2^t} \geq 2^{r-t},
\end{align*} 
where if there is no such $u$, $x$, $y$, $w$, $z$, we set $u = \infty$, $x = \infty$, $y = \infty$, $w = \infty$, or $z = \infty$, respectively.  We then define $\varphi_t(\pi)$ by making the following substitutions among parts of $\pi$,
\[
\begin{cases}
2^{r-a}xd+2^{r} \to (xd+2^{a})^{2^{r-a}} & \text{if }u<\infty, x < \infty, x \leq w, \\
(wd+2^{a})^{2^{r-a}} \to 2^{r-a}wd+2^{r} & \text{if } u<\infty, x > w, w < \infty,\\
2^{r-a}xd + 2^{r} \to (2^{r-a}x-1)d+2^{a - 1}, (2^{a-1})^{\frac{\alpha}{2^{a-1}}} & \text{if }u=\infty, x \leq y, x < \infty, \\
(2^{r-a}y-1)d+2^{a-1},(2^{a-1})^{\frac{\alpha}{2^{a-1}}} \to 2^{r-a}yd+2^{r} & \text{if } u=\infty, x>y, m_{2^{a-1}} \geq \frac{\alpha}{2^{a-1}}, \\
2^{r-a}xd+2^r \to (xd+2^{a})^{2^{r-a}} & \text{if }u=\infty, x>y, x<\infty, m_{2^{a-1}} < \frac{\alpha}{2^{a-1}}, x \leq z, \\
(zd+2^{a})^{2^{r-a}} \to 2^{r-a}zd+2^{r} & \text{if }u =\infty, x>y,m_{2^{a-1}} < \frac{\alpha}{2^{a-1}}, x>z, 
\end{cases}
\]
and setting $\varphi_t(\pi) = \pi$ otherwise.  Under each condition, the number $n$ being partitioned doesn't change.  Since the conditions defining $\varphi_{a}$ are pairwise disjoint, and ensure the required parts exist, $\varphi_{a}$ is well-defined.  As in the previous case, $\varphi_t^{2}(\pi) = \pi$, so we have a bijection, and by our discussion in Remark \ref{rem:parts} as well as \eqref{S_tdef}, we have $\textrm{sgn}(\varphi_t(\pi)) = -\textrm{sgn}(\pi)$ for $\pi \in S_t(n)$.

We again observe that if $x<\infty$, then $\pi$ must satisfy one of the four conditions in the definition of $\varphi_t$.  Thus for any $\pi \not \in S(n)\backslash \bigcup_{j=a-1}^{t-1} S_j(n)$, $x=\infty$, i.e., $\pi$ has no part $\equiv 2^{r} \pmod{2^{r-t}d}$, however may have a part of the form $2^{r-t-1}kd + 2^r$ with $k$ odd.  We also observe that if $\pi \not \in S(n)\backslash \bigcup_{j=a-1}^{t-1} S_j(n)$ and $\pi$ has no part $\equiv  2^{a-1}-d \pmod{2^{r-t}d}$, then $m_{2^{a-1}} < \frac{\alpha}{2^{a-1}}$ and $z = \infty$.
  
We now show that $S^-(n) \subseteq \bigcup_{t=a-1}^{r-2} S_t(n)$.  If $\pi \in S^-(n)$ then $\pi$ must contain a part $\equiv 2^r \pmod{4d}$, since there is an odd number of such parts.  So this part has the form $4dk+2^r$, where $k$ is a nonnegative integer.  If $\pi$ contains such a part with $2^{r-(a-1)}\mid 4k$, then $\pi \in S_{a-1}$.  If not, there exists $a \leq t \leq r-2$ such that $2^{r-t} \parallel 4k$.  Then $\pi \in S_t$.

To finish the proof it now suffices to show that $T(n) \subseteq S \backslash \bigcup_{t=a-1}^{r-2} S_t(n)$.  Let $\pi \in T(n)$.  Then $\pi$ has no parts $\equiv 2^r \pmod{4d}$.  So for each $a-1 \leq t \leq r-2$, in the definition of $\varphi_t$, $x=\infty$.  But in each case, having $\pi \in S_t(n)$ and $x=\infty$ forces $\pi$ to contain a part congruent to either $2^{a-1}$ or $2^a$ modulo an odd multiple of $d$.  Since $\pi$ cannot contain parts of this form, we conclude that $T(n) \not \in S_t(n)$ for any $a-1 \leq t \leq r-2$ as desired.
  
\end{proof}

We are now able to prove Theorem \ref{thm:yeearbitrary}.

\begin{proof}[Proof of Theorem \ref{thm:yeearbitrary}]
By \eqref{eq:r def}, we have that $2^{r-a+1} \leq m+1 < 2^{r-a+2}$.  Thus for each $0\leq n \leq m$, $n$ can be written uniquely in binary as $\sum_{i=0}^{r-a+1}\varepsilon_i  2^i$, where $\varepsilon_i \in \{0,1\}$.  It follows that each term $x^n$ appears in the expansion of $\prod_{i=0}^{r-a+1} (1+x^{2^i})$.  Setting $x=q^{2^{a-1}}$ yields that 
\begin{equation}\label{sub_nonneg}
\prod_{i=0}^{r-a+1} (1+q^{2^{a-1+i}}) - (1 + q^{2^{a-1}} + q^{2\cdot2^{a-1}} + q^{3\cdot2^{a-1}} + \cdots + q^{d})
\end{equation}
has nonnegative $q$-coefficients.  From \eqref{fkg_defs} we thus observe that 
\begin{multline}\label{fk_nonneg}
(1+q^{2^r})\f{a} - \k{a} = 
\left( \prod_{i=0}^{r-a+1} (1+q^{2^{a-1+i}}) - (1 + q^{2^{a-1}} + q^{2\cdot2^{a-1}} + q^{3\cdot2^{a-1}} + \cdots + q^{d}) \right) \\
\cdot (-q^{d+2^{a-1}} ; q^d)_\infty (-q^{d+2^{a}} ; q^d)_\infty \cdots (-q^{d+2^{r-1}} ; q^d)_\infty
\end{multline}
has nonnegative coefficients.  Since 
\[
(1+q^{2^r})\f{a} = \sum_{n\geq 0} \left(\L{a}{n} + \L{a}{n-2^r}  \right) q^n,
\]
where $\L{a}{n}:=0$ for any $n<0$, \eqref{fk_nonneg} shows that when $a,m,n\geq 1$, $d=2^{a-1}m$, and $r$ is defined by \eqref{eq:r def},
\begin{equation}\label{ineq2}
\L{a}{n} + \L{a}{n-2^r}  \geq \Kda.
\end{equation}
Applying Lemmas \ref{yee2.7analog} and  \ref{Lemma2.2analog} to \eqref{ineq2}, we have that when $a\geq 3$, $d=2^{a-1}m$ with $m\geq 8$ and $m \neq 2^{r-a+1} - 1$, where $r$ is defined by \eqref{eq:r def}, then for $n\geq 2^am+2^r+2^{a-1}-1$, 
\[
\q{d}{a}{n} \geq \Gda.
\]
To finish the proof it thus suffices to show that for $a\geq 3$, and $d=2^{a-1}m$ where $m\geq 31$,
\begin{equation}\label{ineq4}
\mathcal{G}_d^{(a)}(2^{a-1}n) \geq \Qdashdash{d}{a}{2^{a-1}n}.
\end{equation}

Setting
\[
S_{a,d} := \{x \in \N \mid x \equiv \pm a \!\!\! \pmod{d+3}\} \backslash \{ a, d+3-a \},
\]
we have $\rho(S_{a,d}; n) =  \Qdadashdash$.  Moreover, since $(-q^{d+2^{r-1}}; q^{2d})_\infty$ has nonnegative $q$-coefficients, it follows that $\Gda \geq \rho(T_{a-1,r-1,d}; n)$, where
\[
T_{a-1,r-1,d}:=\{ y \in \N \mid y \equiv 2^{a-1}, d+2^{a}, ..., d+2^{r-2} \!\!\! \pmod{2d} \}.
\]
Thus, \eqref{ineq4} will follow from showing that $\rho(T_{a-1,r-1,d}; 2^{a-1}n) \geq \rho(S_{a,d}; 2^{a-1}n)$, which is an immediate consequence of Corollary \ref{cor:STgeneral}. 
\end{proof}

\section{Asymptotic results}\label{sec:asymptotics section}

In this section we first prove Theorem \ref{thm:asymptotic result} which is analogous to a result of Andrews \cite[Theorem 2]{andrews_partition_1971}.  We then provide asymptotic bounds for the functions $\qda$ and $\Qda$, and discuss potential methods to computationally prove the finitely many remaining cases of Conjecture \ref{conj:KPconj}. 

\begin{proof}[Proof of Theorem \ref{thm:asymptotic result}]
First, we observe that by equation \eqref{rem:thing}, it suffices to consider $\Delta_d^{(a)}(n)$ instead of $\Delta_d^{(a, -, -)}(n)$.  By work of Meinardus \cite[Theorems \(2\) and \(3\)]{meinardus_uber_1954}, it follows that
\[
\qda \sim C(d ,a) n^{-\frac{3}{4}} e^{2 \sqrt{A n}},
\]
where 
\begin{align}
C(d ,a) &:= \frac{1}{2\sqrt{\pi}}A^{\frac{1}{4}}\left(\alpha^{d + 1 - 2a} \cdot (d(\alpha)^{d - 1} + 1)\right)^{-\frac{1}{2}}, \nonumber \\
A &:= \frac{d }{2} \log^2\alpha + \sum_{r = 1}^{\infty} \frac{(\alpha)^{rd}}{r^2}, \label{eq:Adef}
\end{align}
and $\alpha \in [0,1]$ is the positive real number such that $\alpha^d + \alpha - 1 = 0$.  We thus observe that
\begin{equation}\label{eq:Andrews_q}
\log{\qda} \sim \log{\left(C(d ,a) n^{-\frac{3}{4}}\right)} + 2 \sqrt{A n} \sim 2 \sqrt{A n},
\end{equation}
where by work of Andrews \cite[proof of Theorem 2]{andrews_partition_1971} 
\[
A > \frac{\pi^2}{3d + 9}.
\]
      
Furthermore, from Andrews \cite[Example 1, pg. 97]{andrews_theory_1976} we have
\begin{equation}\label{eq:Andrews_Q}
\Qda \sim \frac{\csc(\frac{\pi a}{d + 3})}{(4\pi) 3^{\frac{1}{4}} (d + 3)^{\frac{1}{4}}} n^{-\frac{3}{4}} e^{2\pi \sqrt{\frac{n}{3d + 9}}},
\end{equation}
for $a < \frac{d + 3}{2}$ such that $a$ is relatively prime to $d + 3$. This gives that
\[
\log{\Qda} \sim \log{\left(\frac{\csc(\frac{\pi a}{d + 3})}{(4\pi) 3^{\frac{1}{4}} (d + 3)^{\frac{1}{4}}} n^{-\frac{3}{4}}\right)} + 2\pi \sqrt{\frac{n}{3d + 9}} \sim 2\pi \sqrt{\frac{n}{3d + 9}}.
\]
        
Together, \eqref{eq:Andrews_q} and \eqref{eq:Andrews_Q} imply that
\[
\lim_{n \to \infty} \left(\log{\qda} - \log{\Qda}\right) = +\infty.
\]
Thus,
\[
\lim_{n \to \infty} \Daa = \lim_{n \to \infty} \qda \left(1 - \frac{\Qda}{\qda}\right) = +\infty,
\]
as desired.
\end{proof}

The asymptotic formulas for \(\qda\) and \(\Qda\) given by Meinardus \cite{meinardus_asymptotische_1954, meinardus_uber_1954} and Andrews \cite{andrews_theory_1976}, while allowing us to examine the asymptotic behavior of \(\Daa\), can be made more explicit by following methods of Alfes, Jameson, and Lemke Oliver \cite{alfes_proof_2011}. 

First we state a theorem of Xia \cite{xia_general_2011}.
\begin{theorem}[Xia \cite{xia_general_2011}, 2011]\label{thm:xia}
Let $d\geq 1$, and $1\leq a < \frac{d+3}{2}$, such that $\gcd(a, d+3)=1$. Then for $n\geq 1$,
\[
\Q{d}{1}{n} \geq \Q{d}{a}{n}.
\]
\end{theorem}

We see that under the hypotheses of Theorem \ref{thm:xia}, $\Q{d}{1}{n}$ bounds $\Qda$ from above, thus the asymptotic expression for $\Q{d}{1}{n}$ obtained by Alfes, Jameson, and Lemke Oliver \cite[Theorem 2.1]{alfes_proof_2011}) is a bound for $\Qda$.  We thus immediately obtain the following.
If $d \geq 4$, $1\leq a < \frac{d+3}{2}$ such that $\gcd(a, d+3)=1$, and $n\geq 1$, then
\begin{equation}\label{eq:Q asymptotic}
Q_d^{(a)}(n) \leq \frac{(3d+9)^{-\frac{1}{4}}}{4\sin \left(\frac{\pi}{d+3}\right)}n^{-\frac{3}{4}}\exp\left( {n^{\frac{1}{2}}\frac{2\pi}{\sqrt{3(d+3)}}}\right) + R(n), 
\end{equation}
where $R(n)$ is an explicitly bounded function given in Alfes, et al \cite[equation \((2.10)\)]{alfes_proof_2011}.

Making only minor modifications to the analysis done by Alfes, et al \cite[Proof of Theorem 3.1]{alfes_proof_2011}, we also get the following asymptotic expression.
Let $A$ and $\alpha$ be defined as in \eqref{eq:Adef}.  Then for $n\geq 1$, 
\begin{equation}\label{eq:q asymptotic}
\qda = \frac{A^{1/4}}{2\sqrt{\pi \alpha^{d+1-2a}(d\alpha^{d-1}+1)}}n^{-3/4}e^{2\sqrt{An}} + r_d(n), 
\end{equation}
where $|r_d(n)|$ can be bounded explicitly.

\subsection{Potential methods for proving remaining cases of Conjecture \ref{conj:KPconj}}
Theorem \ref{thm:final kp thm} resolves Conjecture \ref{conj:KPconj} for $d\geq 62$.  Furthermore, when $d\leq 61$, Kang and Park \cite[Theorem 1.3]{kang_analogue_2020} have resolved the cases when $d=2$ and $d=30$ for even $n$.  Moreover, the cases when $d\leq 61$ is odd and $n$ is odd is trivial, as discussed in the proof of Theorem \ref{thm:final kp thm}. 

It may be possible to use the method of Alfes, et al \cite{alfes_proof_2011} to prove some or all of the remaining finite cases by computation. One idea for even $4\leq d\leq 60$, is to utilize the asymptotic expressions for $\q{d}{2}{n}$ and $\Q{d}{2}{n}$ from \eqref{eq:Q asymptotic} and \eqref{eq:q asymptotic} respectively, to calculate the smallest $n_d$, for fixed $d$, for which these bounds imply $\Da{d}{2}{n} \geq 0$ for all $n \geq n_d$. By \eqref{rem:thing}, this would also imply \(\Ddasha{d}{2}{n} \geq 0\) for all $n \geq n_d$. Then one could conceivably check that $\Ddasha{d}{2}{n} \geq 0$ for all smaller $n$. 

Alternatively, one could utilize the asymptotic bounds for $\q{d}{1}{n}$ and $\Q{d-2}{1}{n}$ given by Alfes, et al \cite{alfes_proof_2011}, and using the same procedure outlined above, potentially prove the missing cases of Proposition \ref{prop:k-2andmlemma}.  Then utilizing the method in the proof of Theorem \ref{thm:final kp thm} this would yield the remaining cases of Conjecture \ref{conj:KPconj}. It is worth noting that this procedure has its own computational challenges.

\section{Discussion of the exclusion of parts}\label{sec:data and dash discussion section}

When considering the values of $\Daa$ for $a\geq 2$ there are many examples where $\Daa < 0$. Kang and Park \cite{kang_analogue_2020} remove the part $d + 3 - a$ from consideration when defining $\Q{d}{2,-}{n}$ to generalize Alder's conjecture to a statement which generalizes the second Rogers-Ramanujan identity. We remove both $a$ and $d+3-a$ when defining $\Q{d}{a,-,-}{n}$ in order to generalize Conjecture \ref{conj:KPconj} for arbitrary $a$. While the removal of these terms may appear unmotivated, we argue that it is natural.  One way to consider this is to determine values of $\qda$ and $\Qda$ when $\Daa$ is negative. In the following we provide some examples.

\begin{example} \label{ex:staircase}
For $a \geq 4$ and $k\geq 0$, we have
\[
\Delta_{a+k-2}^{(a)}(2a+k+1) < 0.
\]
\end{example}

\noindent Example \ref{ex:staircase} follows from observing that $q_{a+k-2}^{(a)}(2a+k+1)=1$ since $(2a+k+1)$ is the only partition counted, and that $Q_{a+k-2}^{(a)}(2a+k+1) \geq 2$ since the partitions $(2a+k+1)$ and $(a^2, k+1)$ are both counted by $Q_{a+k-2}^{(a)}(2a+k+1)$.

\begin{example}\label{ex:cases}
We have that,
\begin{align*}
\Ddash{3a-3}{2a}{2a}{4a} &= -1 \text{ for } a \geq 2, \\
\Ddash{3a-3}{2a}{2a}{6a} &= -1 \text{ for } a \geq 4, \\
\Ddash{5a-3}{2a}{2a}{8a} &= -1 \text{ for } a \geq 4, \\
\Ddash{4a-3}{3a}{3a}{9a} &= -1 \text{ for } a \geq 4,\\
\Ddash{5a-3}{4a}{4a}{12a} &= -1 \text{ for } a \geq 4. 
\end{align*}
\end{example}

\noindent The first line in Example \ref{ex:cases} follows from observing that $q_{3a-3}^{(2a)}(4a)=1$, and that $Q_{3a-3}^{(2a,-)}(4a) = 2$ since it counts the partitions $(4a)$ and $((2a)^2)$.  The other four follow in the same manner.

By examining the cases in Examples \ref{ex:staircase} and \ref{ex:cases}, we find that the parts which yield more partitions for $\Qda$ than $\qda$ are precisely the parts equal to $d + 3 - a$ and $a$.  Roughly speaking, when $n$ is less than $2a + d$, $\qda$ is bounded at $1$. However the $a$ and $d + 3 - a$ parts are small enough that they can contribute to partitions of $n$ counted by $\Qda$, especially when $n$ is a multiple of $a$. The removal of these parts for large $n$ is unnecessary to maintain the nonnegativity of $\Daa$, which is further evidenced by the proof of Theorem \ref{thm:asymptotic result}.  A further justification for removing these parts is that we observe computationally that $\Daadashdash \geq 0$ for $a \geq 3$.



\begin{thebibliography}{10}

\bibitem{alder_research_1956}
H.~L. Alder.
\newblock Research {Problems}.
\newblock {\em Bulletin of the American Mathematical Society}, 62(1):76,
  January 1956.

\bibitem{alder_nonexistence_1948}
Henry~L. Alder.
\newblock The nonexistence of certain identities in the theory of partitions
  and compositions.
\newblock {\em Bulletin of the American Mathematical Society}, 54(8):712--723,
  August 1948.

\bibitem{alfes_proof_2011}
Claudia Alfes, Marie Jameson, and Robert J.~Lemke Oliver.
\newblock Proof of the {Alder}-{Andrews} conjecture.
\newblock {\em Proceedings of the American Mathematical Society},
  139(01):63--78, January 2011.

\bibitem{andrews_partition_1971}
George Andrews.
\newblock On a partition problem of {H}. {L}. {Alder}.
\newblock {\em Pacific Journal of Mathematics}, 36(2):279--284, February 1971.

\bibitem{andrews_general_1969}
George~E. Andrews.
\newblock A {General} {Theorem} on {Partitions} with {Difference} {Conditions}.
\newblock {\em American Journal of Mathematics}, 91(1):18, January 1969.

\bibitem{andrews_theory_1976}
George~E. Andrews.
\newblock {\em The theory of partitions}.
\newblock Cambridge Mathematical Library. Cambridge University Press,
  Cambridge, 1998.
\newblock Reprint of the 1976 original.

\bibitem{kang_analogue_2020}
Soon-Yi Kang and Eun~Young Park.
\newblock An analogue of {Alder}-{Andrews} {Conjecture} generalizing the 2nd
  {Rogers}-{Ramanujan} identity.
\newblock {\em Discrete Mathematics}, 343(7), July 2020.

\bibitem{lehmer_two_1946}
D.~H. Lehmer.
\newblock Two nonexistence theorems on partitions.
\newblock {\em Bulletin of the American Mathematical Society}, 52(6):538--545,
  June 1946.

\bibitem{meinardus_asymptotische_1954}
Gunter Meinardus.
\newblock Asymptotische {Aussagen} \"uber {Partitionen}.
\newblock {\em Mathematische Zeitschrift}, 59:388--398, 1954.

\bibitem{meinardus_uber_1954}
Gunter Meinardus.
\newblock \"{U}ber {Partitionen} mit {Differenzenbedingungen}.
\newblock {\em Mathematische Zeitschrift}, 61:289--302, 1954.

\bibitem{xia_general_2011}
Li-meng Xia.
\newblock A general estimation for partitions with large difference: {For} even
  case.
\newblock {\em Journal of Number Theory}, 131(12):2426--2435, December 2011.

\bibitem{yee_partitions_2004}
A.~J. Yee.
\newblock Partitions with difference conditions and {Alder}'s conjecture.
\newblock {\em Proceedings of the National Academy of Sciences},
  101(47):16417--16418, November 2004.

\bibitem{yee_alders_2008}
Ae~Ja Yee.
\newblock Alder's conjecture.
\newblock {\em Journal fur die reine und angewandte Mathematik (Crelles
  Journal)}, 2008(616), January 2008.

\end{thebibliography}
\end{document}